\documentclass[a4paper,10pt]{amsart}
\usepackage{txfonts,amsfonts}
\usepackage[arrow,matrix]{xy}
\usepackage{amsmath,amssymb,amscd,bbm,amsthm,mathrsfs,dsfont,bm}
\usepackage{extarrows}
\usepackage{latexsym}
\usepackage{graphicx}
\usepackage{color}
\usepackage{xy}
\usepackage{leftidx}
\input xy
\input xypic
\xyoption{all}

\numberwithin{equation}{section}

\theoremstyle{plain}
\newtheorem{theorem}{Theorem}[section]
\newtheorem{corollary}[theorem]{Corollary}
\newtheorem{lemma}[theorem]{Lemma}

\newtheorem{proposition}[theorem]{Proposition}

\theoremstyle{definition}
\newtheorem{definition}[theorem]{Definition}

\theoremstyle{remark}
\newtheorem{remark}[theorem]{Remark}

\DeclareMathOperator{\lw}{\rm LW}

\DeclareMathOperator{\Hom}{Hom}

\DeclareMathOperator{\gkdim}{\rm GKdim }
\DeclareMathOperator{\id}{\rm id }
\DeclareMathOperator{\sh}{\rm Sh }

\DeclareMathOperator{\lex}{\rm lex }
\DeclareMathOperator{\glex}{\rm glex }

\DeclareMathOperator{\gr}{\rm gr }

\newcommand{\B}{\mathcal{B}}
\newcommand{\C}{\mathcal{C}}
\newcommand{\N}{\mathcal{N}}

\renewcommand{\L}{\mathbb{L}}

\begin{document}

\title{The structure of connected (graded)  Hopf algebras}

\author{G.-S. Zhou,\; Y. Shen\; and \; D.-M. Lu\ }

\address{\rm Zhou \newline \indent
School of Mathematics and Statistics, Ningbo University, Ningbo 315211, China
\newline \indent E-mail: zhouguisong@nbu.edu.cn
\newline\newline
\indent Shen \newline
\indent Department of Mathematics, Zhejiang Sci-Tech University, Hangzhou 310018, China
\newline \indent E-mail: yuanshen@zstu.edu.cn
\newline\newline
\indent Lu \newline
\indent School of Mathematics, Zhejiang University, Hangzhou 310027, China
\newline \indent E-mail: dmlu@zju.edu.cn}

\begin{abstract}
In this paper, we establish a structure theorem for connected graded Hopf algebras over a field of characteristic $0$ by claiming the existence of a  family of homogeneous generators and a total order on the index set  that satisfy some desirable conditions. The approach to the structure theorem is constructive, based on
the combinatorial properties  of Lyndon words and the standard bracketing on words. As a surprising consequence of the structure theorem, we show that connected graded Hopf algebras of finite Gelfand-Kirillov dimension over a field of characteristic $0$ are all iterated Hopf Ore extensions of the base field. In addition,  some keystone facts of connected Hopf algebras over a field of characteristic $0$ are observed as corollaries of the structure theorem, without the assumptions of having finite Gelfand-Kirillov dimension (or affineness) on Hopf algebras or of that the base field is algebraically closed.
\end{abstract}




\maketitle





\section*{Introduction}

\newtheorem{maintheorem}{\bf{Theorem}}
\renewcommand{\themaintheorem}{\Alph{maintheorem}}
\newtheorem{maincorollary}[maintheorem]{\bf{Corollary}}
\renewcommand{\themaincorollary}{}
\newtheorem{mainquestion}[maintheorem]{\bf{Question}}
\renewcommand{\themainquestion}{}

The classification project of noetherian/affine  Hopf algebras of finite Gelfand-Kirillov dimension (GK dimension, for short)  by using homological tools has attracted many recent studies and interests. The cases of GK dimensions one and two have witnessed much progress by Brown, Goodearl, Liu, Zhang and their collaborators \cite{BZ1, GZ, Liu1, Liu2, WZZ0, Liu3}.
Some mild conditions are added on Hopf algebras in all these works. For the cases of higher GK dimensions, one line of research is by adding on Hopf algebras a quite restrictive condition: the Hopf algebras are \emph{connected}, i.e., the coradical is one-dimensional. This line of research was initiated by Zhuang \cite{Zh}  and  continued in \cite{BG, BGZ, BOZZ, WZZ}.
As a matter of fact, connected Hopf algebras of GK dimensions three and four over algebraically closed fields of characteristic $0$ are completely classified by Zhuang \cite{Zh} and Wang, Zhang and Zhuang \cite{WZZ} respectively. Some new interesting Hopf algebras occur in the classification results. The motivation of this paper is to understand the structure of general connected  Hopf algebras.

There is little chance to list all isomorphism classes of connected Hopf algebras. However, it is well-known that commutative affine connected Hopf algebras over algebraically closed fields of characteristic $0$ are polynomial algebras in finitely many variables, and  cocommutative connected Hopf algebras over fields of characteristic $0$ are universal enveloping algebras of Lie algebras. Both facts play important roles in the study of general connected Hopf algebras. One goal of this paper is to describe to some extent, some other natural subclasses of connected Hopf algebras. We examine this idea by considering the subclass of connected graded Hopf algebras.

Recall that a  \emph{graded Hopf algebra} is a Hopf algebra $H$ equipped with a grading $H=\bigoplus_{n\geq 0}H_n$  such that  $H$ is simultaneously a graded algebra and a graded coalgebra,  and the antipode  preserves the given grading. Such a graded Hopf algebra is called \emph{connected} if $H_0$ is one-dimensional.   Clearly,  connected graded Hopf algebras are connected Hopf algebras. Note that the class of connected graded Hopf algebras properly contains the class of  universal enveloping algebras of positively graded Lie algebras.  Actually, there is a connected graded Hopf algebra of GK dimension $5$ which is  not  isomorphic even as an algebra to the universal enveloping algebra of any Lie algebra \cite{BGZ}. We have:

\begin{maintheorem}(Theorem \ref{structure-PBW-generator})\label{main-structure-PBW-generator}
Assume that the base field $k$ is of characteristic $0$. Let $H$ be a connected graded Hopf algebra. Then there exists an indexed family $\{z_\gamma\}_{\gamma\in \Gamma}$ of homogeneous elements of $H$ of positive degrees and a total order $\leq$ on  $\Gamma$ satisfying the following conditions:
\begin{enumerate}
\item for every index  $\gamma \in \Gamma$,
\[
\Delta_H(z_\gamma) \in 1\otimes z_\gamma +z_\gamma\otimes 1 + \bigoplus_{i,j>0,~ i+j = \deg(z_\gamma)} (H^{<\gamma})_i \otimes (H^{<\gamma})_j,
\]
where $H^{<\gamma}$ denotes the subalgebra of $H$ generated by $\{~ z_\delta~|~ \delta \in \Gamma,~ \delta<\gamma ~\}$;
\item  for every pair of indexes
$\gamma, \delta \in \Gamma$ with $\delta<\gamma$,
$$z_\gamma z_\delta -z_\delta z_\gamma \in (H^{<\gamma})_{\deg(z_\gamma z_\delta)};$$
\item the set $\{~ z_{\gamma_1} \cdots z_{\gamma_n} ~|~ n\geq0,~ \gamma_1,\cdots, \gamma_n \in \Gamma, ~ \gamma_1\leq \cdots \leq \gamma_n ~\}$ is a basis of $H$.
\end{enumerate}
Moreover, $\gkdim H$ is finite if and only if $\Gamma$ is finite; and in this case $\gkdim H = \#(\Gamma)$.
\end{maintheorem}

As corollaries of the above theorem, some interesting properties of connected Hopf algebras over a field of characteristic $0$ are established, without the assumptions of having finite GK dimension (or affineness) on Hopf algebras or that the base field is algebraically closed.
For example,  we show that a commutative connected Hopf algebra over a field of charactersitic $0$  is isomorphic as an algebra to the polynomial algebra in some family of variables (Propositon \ref{connected-commutative-Hopf-algebra}). In addition,  Theorem \ref{main-structure-PBW-generator}  has a corollary (Theorem \ref{structure-connected-Hopf}) which may help us find new connected Hopf algebras of finite GK dimension over a field of characteristic $0$.

The most surprising consequence of Theorem \ref{main-structure-PBW-generator} is on the relation of connected Hopf algebras and Ore extensions. Recall from \cite{BOZZ} that a Hopf algebra $H$ is called an \emph{iterated Hopf Ore extension of $k$} (IHOE, for short)  if there is a  chain of Hopf subalgebras $k= H^0 \subset H^1 \subset \cdots \subset H^n =H$ with each of the extensions $H^i\subset H^{i+1}$ being  an Ore extension as algebras. On  the one hand, it is known from the classification result that all connected Hopf algebras of GK dimension $\leq 4$ are IHOEs when the base field is  algebraically closed of characteristic $0$ (\cite[Subsection 3.4]{BOZZ}). On the other hand, not all connected Hopf algebras of finite GK dimension are IHOEs. Counterexamples occur in the class of universal enveloping algebras of semisimple Lie algebras (\cite[Example 3,1 (iv)]{BOZZ}). The following theorem provides a large class of connected Hopf algebras  that are IHOEs.

\begin{maintheorem}(Theorem \ref{structure-connected-graded-Hopf})
\label{main-structure-connected-graded-Hopf}
Assume that the base field $k$ is of characteristic $0$. Let $H$ be a connected graded Hopf algebra  of finite GK dimension $d$. Then $H$ is an IHOE. More precisely, there is a finite sequence $z_1,\cdots, z_d $ of homogeneous elements of $H$  of positive degrees such that
$$
H = k[z_1][z_2;\delta_2]\cdots [z_d; \delta_d]
$$
and
$$
\Delta_H(z_r) \in 1\otimes z_r + z_r\otimes 1 + \sum_{i,j>0, ~ i+j =\deg(z_r)}(H^{\leq r-1})_i \otimes (H^{\leq r-1})_j, \quad r=1,\cdots,d,
$$
where $H^{\leq 0}=k$, $H^{\leq i}$ is the subalgebra of $H$ generated by $z_1,\cdots, z_i$ for $i=1,\cdots,d$
and  $\delta_i$ a derivation of $H^{\leq i-1}$ for $i=2,\cdots, d$.
In particular, $H^{\leq 0}, H^{\leq1}, \cdots, H^{\leq d}$
are all graded Hopf subalgebras of $H$.
\end{maintheorem}

Note that each  Hopf Ore extension $H^{\leq i} = H^{\leq i-1}[z_i;\delta_i]$  in the above theorem  has no twisting by any nontrivial automorphism. This phenomenon does not occur for general connected Hopf algebras.
In view of the above theorem, one may ask the following question in the spirit of \cite{BGZ}.

\begin{mainquestion}\label{question-IHOE}
Let $H$ be a Hopf algebra over a field of characteristic $0$. Assume that $H$ is of finite GK dimension and connected graded as an algebra. Then is $H$ an IHOE?
\end{mainquestion}

Since IHOEs are connected Hopf algebras (Corollary \ref{IHOE-connected}), a positive answer of the above question  implies a positive answer of  \cite[Question 0.4]{BGZ}, where the authors ask: If a Hopf algebra of finite GK dimension is connected graded as an algebra then is it a connected Hopf algebra?

Now let us briefly summarize the basic ideas used in the  proof of Theorem \ref{main-structure-PBW-generator}. The approach is constructive. First choose an arbitrary set $X$ of homogeneous generators of $H$ and then fix an appropriate well-order on $X$, and let $k\langle X\rangle$ denote the free ($k$-)algebra on $X$. By the combinatorial properties of Lyndon words and the standard bracketing on words, one may construct from $X$ a new family of homogeneous generators of $H$ indexed by a totally ordered set (Lemma \ref{standard-basis-2}).   Next we employ the coalgebra structure of $H$. It is easy to lift  $\Delta_H$ into a graded algebra homomorphism $\Delta: k\langle X\rangle \to  k\langle X\rangle \otimes  k\langle X\rangle$ that  is triangular  (see Definition \ref{definition-triangular-comultiplication}).  Generally, $\Delta(x) \neq 1\otimes x+x\otimes 1$ because $x$ is not necessarily primitive in $H$; and moreover $\Delta$ is not necessarily coassociative in the sense that $(\Delta \otimes \id ) \circ \Delta = (\id \otimes \Delta) \circ \Delta.$  Nevertheless, as an immediate consequence of  a  technical result (Proposition \ref{quasi-Lie-imply-LK}), the triangularity of $\Delta$ is sufficient to deduce that this new  family of  generators and the total order on the index set satisfy the requirements.

The technical result (Proposition \ref{quasi-Lie-imply-LK}) has its own interest. Ideals of free algebras that satisfy the conditions listed there are studied extensively in the subsequent paper \cite{ZSL3}.
The proof of the technical result is complicated and occupies a large part of the paper. However, the idea of the proof has its origin in \cite{Kh, Uf}, where Lyndon words and the braided bracketing  on words are employed  to construct  a Poincar\'{e}-Birkhoff-Witt type basis  for  primitively generated braided Hopf algebras.

The paper is organized as follows.  In Section 1,
we recall the definitions and basic facts of Lyndon words and the standard bracketing on words.  In Section 2,
we introduce the notion of triangular comultiplication on free algebras and study their properties.  Section 3 is the main part of the paper. It is devoted to study the structure of connected (graded) Hopf algebras. In particular, the above two main theorems  are proved there. As a matter of fact, though the terminologies and results developed in sections 1 and 2 have their own interest, they are aimed to prove Theorem \ref{main-structure-PBW-generator} (i.e. Theorem \ref{structure-PBW-generator}). The last section is  devoted to prove the technical result (Proposition \ref{quasi-Lie-imply-LK}).

Throughout the paper, we work over a fixed field denoted by $k$. All vector spaces, algebras, coalgebras and unadorned tensors  are over $k$. The notation $\mathbb{N}$ denotes the set of non-negative integers.

\section{Lyndon words and the standard bracketing}
\label{subsection-Lyndon}

The bricks in the construction of the required family of generators for connected graded Hopf algebras as stated in the main theorems are Lyndon words and the standard bracketing on words. In this section, we recall
 the definitions of Lyndon words and the standard bracketing.  Moreover we present some of their well-known properties for the reader's convenience. The  materials regarding Lyndon words and standard bracketing in our context are mostly extracted or modified  from \cite{GF, Kh, Lo}.

Throughout this section, $X$ stands for a well-ordered alphabet.  If $X=\{x_1,\cdots, x_\theta\}$ for some positive integer $\theta$, then $X$ is tacitly equipped with the natural order $x_1<\cdots <x_\theta$.
We denote by $\langle X\rangle $ the set of all words on $X$. The empty word is denoted  by $1$.  The \emph{length} of a word $u$ is denoted by $|u|$.

A word $v$ on $X$ is called a \textit{factor}  of a word $u$ on $X$  if there exist words $w_1,w_2$ on $X$ such that  $w_1vw_2=u$. If $w_1$ (resp. $w_2$) can be chosen to be the empty word then $v$ is called a \emph{prefix} (resp. \emph{suffix}) of $u$.  A  factor (resp. prefix, resp. suffix) $v$ of $u$ is called  \textit{proper} if $v\neq u$.

The \emph{lexicographic order} on $\langle X\rangle$, denoted by $<_{\rm lex}$, is defined as follows. For words $u,v\in \langle X\rangle$,
\begin{eqnarray*}
\label{definition-deglex}
u <_{\rm lex} v\,\, \Longleftrightarrow \,\, \left\{
\begin{array}{llll}
v \text{ is a  proper prefix of } u,\quad \text{or}&& \\
u=rxs,\, v=ryt,\, \text{ with } x,y\in X;\, x<y;  \, r,s,t\in \langle X\rangle.
\end{array}\right.
\end{eqnarray*}
Clearly, it is a total order compatible with the concatenation of words from left but not from right.  For example if $x,y\in X$ with $x>y$, one has $x>_{\rm lex}x^2$ but $xy<_{\rm lex} x^2y$.

\begin{lemma}\label{lex-order}
 If $u<_{\rm lex}v$ and $v$ is not a prefix of $u$,  then $uw<_{\rm lex}vw'$ for all words $w,w'$.
\end{lemma}
\begin{proof}
It is clear from the definition of the lexicographic order.
\end{proof}

\begin{definition}
A word $u\in \langle X\rangle$ is called \textit{Lyndon}  if $u$ is nonempty and $u>_{\rm lex} wv$ for every factorization $u=vw$ with $v,w$ nonempty. The set of all Lyndon words on $X$ is denoted by
$\L=\mathbb{L}(X)$.
\end{definition}

\begin{remark}
In \cite{GF,Lo,LR}, $u<_{\lex}v$ means that either $u$ is a proper prefix of $v$ or  $u=rxs$ and $v=ryt$ with $x<y$; and a word $u\in \langle X\rangle$ is Lyndon if $u\neq1$ and $u<_{lex} wv$ for every factorization $u=vw$ with $v,w\neq1$. So the results in \cite{GF,Lo,LR} need to be changed accordingly in our context. We follow the convention of \cite{Kh}, where Lyndon words are called \emph{standard words} after Shirshov.
\end{remark}

For a word  $u$ of length $\geq2$, define   $u_R\in \langle X\rangle $ to be  the lexicographically largest proper suffix of $u$ and define  $u_L\in \langle X\rangle $ by the decomposition $u=u_Lu_R$.  The pair of words
$$\sh(u):=(u_L,u_R)$$ is called  the {\em Shirshov factorization} of $u$. As an example,  $\sh(x_2^2x_1x_2x_1)=(x_2^2x_1,x_2x_1)$.

Some facts concerning Lyndon words used in the sequel are as follows.
\begin{proposition}\label{fact-Lyndon} \hfill
\begin{enumerate}
\item[(L1)] (\cite[Lemma 2]{Kh}) Let $u$ be a nonempty word. Then $u$ is Lyndon if and only if $u$ is lexicographically greater than any one of  its proper nonempty suffices.
\item[(L2)](\cite[Lemma 4]{Kh}) Let  $u_1>_{\rm lex}u_2>_{\rm lex} u'$ be nonempty words. If  $u_1u_2$ and $u'$ are Lyndon words, then $u_1u_2u' >_{\rm lex} u_1 u' >_{\rm lex} u'$ and $u_1u_2u' >_{\rm lex} u_2u' >_{\rm lex} u'$.
\item[(L3)] (\cite[Proposition 5.1.3]{Lo}) Let $u$ be word of length $\geq2$. Then $u$ is a Lyndon word if and only if $u_L$ and $u_R$ are both Lyndon words and $u_L>_{\rm lex} u_R$.
\item[(L4)](\cite[Proposition 5.1.4]{Lo}) Let $u,v$ be Lyndon words such that $u >_{\rm lex} v$. Then $\sh(uv)=(u,v)$ if and only if either $u$ is a letter or  $u$ is of length $\geq 2$ with $u_R\leq_{\rm lex} v$.
\item[(L5)](\cite[Theorem 5.1.5]{Lo}) Every nonempty word $u$ can be written uniquely as a (lexicographically) nondecreasing product $u=u_1u_2\cdots u_r$ of Lyndon words.
\end{enumerate}
The unique decomposition of a word $u$ as in (L5) is called the \emph{Lyndon decomposition} of $u$.   \hfill $\Box$
\end{proposition}

\begin{lemma}(\cite[Lemma 5]{Kh})\label{lex-order-Lyndon}
Let $u=u_1\cdots u_m$ and $v=v_1\cdots v_n$ be two nonempty words in Lyndon decomposition. Then $u<_{\lex} v$ if and only if either $n< m$ and $u_i= v_i$ for $i\leq n$, or there is an integer $l \leq \min\{m, n\}$ such that $u_i=v_i$ for $i< l$ but $u_l <_{\lex } v_l$. \hfill $\Box$
\end{lemma}

We denote by  $k\langle X\rangle$ the free algebra on $X$, which has all words on $X$ as a linear basis. Elements of $k\langle X\rangle $ are frequently  called (noncommutative) polynomials on $X$ (with coefficients in $k$).
Let $[-,-]: k\langle X\rangle \times k\langle X\rangle \to k\langle X\rangle$ be the commutator operation on $k\langle X\rangle$ defined by $$[f,g] = fg-gf.$$ The \emph{standard bracketing} on $\langle X\rangle$ is the map $[-]: \langle X\rangle \to k\langle X\rangle$ defined as follows.  First set $[1] =1$ and  $[x]:=x$ for  $x\in X$; and then for  words $u$ of length $\geq 2$, inductively  set
\begin{eqnarray*}
[u]=
\left\{
\begin{array}{ll}
~ [[u_L], [u_R]], & u \text{ is Lyndon},  \\
~ [u_L]  [u_R], &  u \text{ is not Lyndon}.
\end{array}
\right.
\end{eqnarray*}
Note that $[u_1u_2\cdots u_n] = [u_1][u_2]\cdots [u_n]$ for Lyndon words $u_1\leq_{\rm lex}u_2\leq_{\rm lex} \cdots \leq_{\rm lex} u_n$.

\begin{lemma}(\cite[Lemma 5]{Kh})\label{bracketing-leading}
Let $u$ be a word on $X$. Then $[u]-u$  is a linear combination of the words lexicographically less than $u$ and
having the same occurrences of each letter as $u$.   \hfill $\Box$
\end{lemma}

\begin{lemma}(\cite[Lemma 6]{Kh})\label{bracketing-expansion}
Let $u >_{\lex} v$ be Lyndon words. Then $[[u], [v]]$ is a linear combination of  polynomials of the form $[w_1]\cdots [w_n]$ with $w_1, \cdots, w_n$ Lyndon words satisfying that $v <_{\lex} w_1,\cdots, w_n\leq_{\lex} uv$ and $w_1\cdots w_n$ has the same occurrences of each letter as $uv$.  \hfill $\Box$
\end{lemma}

\begin{lemma}(\cite[Lemma 7]{Kh})\label{reordering-bracketing}
For any nonempty finite sequence $u_1,\cdots, u_n$ of Lyndon words,  the product $[u_1]\cdots [u_n]$ is a linear combination of  the standard bracketing of  words that are  nondecreasing products of Lyndon words that  lie between the smallest and greatest Lyndon words among $u_1,\cdots, u_n$, and that
have the same occurrences of each letter as $u_1\cdots u_n$. \hfill $\Box$
\end{lemma}

For a word $w$ on $X$, we write $k\langle X\rangle^{<w}$ (resp. $k\langle X\rangle^{\leq w}$) for  the  subalgebra of $k\langle X\rangle$ generated by
$$
\left\{\,
[u]\;|\; u<_{\lex}w, u\in \L
\right\}
\quad (\text{resp. }\left\{\,
[u]\;|\; u\leq_{\lex}w, u\in \L
\right\}).
$$

\begin{lemma}\label{commutator-subalgebra}
For every pair of Lyndon words $u>_{\lex} v$ on $X$,
\[
[[u], k\langle X\rangle^{< v}] \subseteq k\langle X\rangle^{< uv} \quad \text{and} \quad [[u], k\langle X\rangle^{\leq v}] \subseteq k\langle X\rangle^{\leq uv}.
\]
\end{lemma}

\begin{proof}
We only show the first inclusion, the second one is similar to see. Let $w_1, \cdots, w_r$ be a finite sequence of Lyndon words that $<_{\lex} v$. One has
\[
[[u], [w_1]\cdots [w_r]] = \sum_{i=1}^{r} [w_1]\cdots [w_{i-1}] \cdot  [[u], [w_i] ] \cdot [w_{i+1}] \cdots [w_r].
\]
By Proposition \ref{fact-Lyndon} (L1), $w_i<_{\lex} v <_{\lex} uv$; and by Lemma \ref{bracketing-expansion}, $$[[u], [w_i]] \in k\langle X\rangle^{\leq uw_i} \subseteq k\langle X\rangle^{< uv}$$ because $uw_i<_{\lex} uv$. The result follows immediately.
\end{proof}

Now we assume that  the free algebra $k\langle X\rangle$ is connected graded with each letter homogeneous of positive degree. The degree of a homogeneous polynomial $f$ is denoted by $\deg(f)$.  Define the \emph{graded lex order} on $\langle X\rangle$, denoted by $<_{\glex}$, as follows. For words $u, \, v$,
\begin{eqnarray*}
\label{definition-deglex}
u <_{\rm \glex} v\,\, \Longleftrightarrow \,\, \left\{
\begin{array}{llll}
\deg(u)<\deg(v), \quad \text{or} && \\
\deg(u) =\deg(v)\, \text{ and } u<_{\lex} v.
\end{array}\right.
\end{eqnarray*}
Clearly, it is a well order compatible with the  concatenation of words from both sides.  The \emph{leading word} of a nonzero polynomial $f$, denoted by $\lw(f)$, is the largest word that occurs in $f$   with respect to the graded lex order. Note that $\lw([w]) =w$ for every word $w$ by Lemma \ref{bracketing-leading}.

Let $I$ be an ideal of $k\langle X\rangle$. A word on $X$ is called \emph{$I$-reducible}  if it is the leading word of a polynomial in $I$. A word that is not $I$-reducible is called \emph{$I$-irreducible}. We let
\begin{align*}
&\N_I:=\{~ u\in \L\; |\; u \text{ is } I\text{-irreducible}~\}, \\
&\B_I:=\{~ u_1\cdots u_n\; |\; u_1\leq_{\lex}\cdots\leq_{\lex}u_n, u_1,\cdots,u_n\in \N_I ~ \}.
\end{align*}
By Proposition \ref{fact-Lyndon} (L5), it is  easy to see  that $\B_I$ contains every $I$-irreducible word.

\begin{remark}
Let $X=\{x_1,x_2,x_3\}$ and let $f=[x_2x_1]-x_3=x_2x_1-x_1x_2-x_3$. If $k\langle X\rangle$ is connected graded by $\deg(x_1)=\deg(x_2)=\deg(x_3) =1$ then $\lw(f) =x_2x_1$; but if $k\langle X\rangle$ is connected graded by $\deg(x_1)=\deg(x_2)=1$ and $\deg(x_3) =2$ then $\lw(f) =x_3$. So for a given ideal $I$, $\N_I$ may include $X$ with respect to one grading structure on $k\langle X\rangle$ but does not with respect to another.
\end{remark}

\begin{lemma}\label{standard-basis-2}
Let $I$ be an ideal of $k\langle X\rangle$ and $A:= k\langle X\rangle/I$. For every integer $n\geq 0$, the residue classes  of the standard bracketing of  $I$-irreducible words of degree $\leq n$ form a basis of the subspace  $F_nA:=(k\langle X\rangle_{\leq n} +I)/ I$ of   $A$. Consequently, the residue classes of the standard bracketing of $I$-irreducible words form a basis of $A$. In particular, $A$ is generated by $\{~ [u]+I ~|~u \in \N_I ~\}$.
\end{lemma}

\begin{proof}
We prove the first statement, the others are clear.
Firstly, we show that these residue classes are linearly independent. Suppose not, then there exists a polynomial  $f=\sum_{i=1}^r \lambda_i [u_i] \in I$, with $u_i$ pairwise distinct $I$-irreducible words of degree $\leq n$ and $\lambda_i\in k\backslash\{0\}$. We may assume that $u_1> _{\glex} \cdots >_{\glex} u_r$. By Lemma \ref{bracketing-leading}, the leading word of $f$ is $u_1$, which is impossible.

Now we show these residue classes span $F_nA$.  It suffices to show they span the residue classes of all words of degree $\leq n$. We show this by induction on words with respect to the graded lex  order. Clearly, it is true for the empty word, which is the smallest element with respect to $\leq_{\glex}$. Let $u$ be a nonempty word of degree $\leq n$. If $u$ is $I$-reducible, then $u+I= f_u+I $ with $f_u$ a linear combination of words that $<_{\glex} u$; and if $u$ is $I$-irreducible, then $u+I= ([u]+I) + (f_u +I)$ with  $f_u:=u-[u]$, which  by Lemma \ref{bracketing-leading} is also a  linear combination of words which are $<_{\glex} u$. So by the induction hypothesis,  $u+I$ is a linear combination  of  the residue classes of the standard bracketing of  $I$-irreducible words of degree $\leq n$.
\end{proof}

\section{Triangular comultiplications on free algebras}
\label{section-proof-structure-PBW-generator}

Let $k$ be a field of arbitrary characteristic. Given an algebra $A$ and an algebra homomorphism $\Delta_A:A\to A\otimes A$,  one may choose a set of generators $X$ and then lift $\Delta_A$ into an algebra homomorphism $\Delta: k\langle X\rangle \to  k\langle X\rangle \otimes  k\langle X\rangle$. Generally, it is not easy to read from $\Delta$ any interesting information for the ideal $I \subseteq k\langle X\rangle$ of relations of $A$. In this section, we introduce a class of algebra homomorphisms $k\langle X\rangle \to  k\langle X\rangle \otimes  k\langle X\rangle$ and study their properties by the combinatorial properties of Lyndon words and the standard bracketing. An interesting observation is that if $\Delta$ can be chosen to belong this class then the ideal $I$  of relations behaves well. This observation is the key to explore the structure of connected graded Hopf  algebras in the next section.

Throughout, $X$ stands for a well-ordered alphabet and $k\langle X\rangle $ is connected graded with each letter homogeneous of positive degree. The tensor algebra $k\langle X\rangle \otimes k\langle X\rangle$ is graded in the natural way. For a word $w$, recall that $k\langle X\rangle^{<w}$ (resp. $k\langle X\rangle^{\leq w}$) denotes the homogeneous subalgebra of $k\langle X\rangle$ generated by the set of the standard bracketing of Lyndon words which are $<_{\lex} w$ (resp. $\leq_{\lex} w$).

\begin{definition}\label{triangular-comultiplication}\label{definition-triangular-comultiplication}
A \emph{comultiplication} on $k\langle X\rangle$ is  an algebra homomorphism  $\Delta: k\langle X\rangle  \to k\langle X\rangle \otimes k\langle X\rangle$.  A comultiplication $\Delta$ is called \emph{triangular}  if  for each letter $x\in X$,
$$
\Delta(x) ~ \in ~  1\otimes x + x\otimes 1 + \sum_{ i, j  ~> ~0 \atop i+j = \deg (x)} k\langle X\rangle_i^{< x} \otimes k\langle X \rangle_j^{<x} + \big (k\langle X\rangle \otimes k\langle X\rangle \big )_{<\deg(x)}.
$$
If in addition $\Delta$ is graded linear of degree $0$ then it is called \emph{graded triangular}.
\end{definition}

\begin{remark}\label{remark-standard-comul}
The \emph{standard comultiplication} on $k\langle X\rangle$ is defined to be the algebra homomorphism
$$
\Delta_s: k\langle X\rangle  \to k\langle X\rangle \otimes k\langle X\rangle, \quad X \ni x\mapsto 1\otimes x +x\otimes 1.
$$
 It  is  graded triangular by definition. Note that $\Delta_s$ is coassociative,
which clearly does not necessarily  hold for general (triangular) comultiplications.  Moreover, it is easy to check that
\begin{equation}\label{ST}
\Delta_s([u]^n) =   \sum_{p=0}^n \binom{n}{p} ~  [u]^p \otimes [u]^{n-p} \tag{ST}
\end{equation}
for every Lyndon word $u$ and every positive integer $n$.
\end{remark}

\begin{proposition}\label{comultiplication}
Let $\Delta$ be a triangular comultiplication on $k\langle X\rangle$.
Then
 \begin{eqnarray*}
 \Delta([u]^n)  &\in &  \sum_{p=0}^n \binom{n}{p} ~  [u]^p \otimes [u]^{n-p}
   +  \sum_{r, s ~ \geq~ 0 \atop r+s ~<~ n } \sum_{ i, j  ~>~ 0  \atop i+j = (n-r-s) \deg (u)} k\langle X \rangle_i^{< u}\cdot  [u]^r   \otimes k\langle X\rangle_j^{< u} \cdot [u]^s \\
   && + \Big (k\langle X\rangle \otimes k\langle X\rangle \Big )_{< \deg(u^n)}
 \end{eqnarray*}
 for every Lyndon word $u$  and every positive integer $n$.
\end{proposition}

\begin{proof}
We show the result by induction on $n$. For $n=1$, we do it by induction on  $l=|u|$. For $l=1$, $u$ is a letter and the formula holds by the assumption.
Assume $l>1$. By induction,
\begin{eqnarray*}
\Delta([u])  &=& [\Delta([u_L]), \Delta([u_R])] \\
&\in& 1\otimes [u] +[u]\otimes 1 \\
&&  +    \sum_{ i, j  ~>~ 0  \atop i+j = \deg (u_R)} [[u_L],k\langle X \rangle_i^{< u_R} ] \otimes k\langle X\rangle_j^{< u_R} +   \sum_{ i, j  ~>~ 0  \atop i+j = \deg (u_R)}  k\langle X\rangle_i^{< u_R} \otimes [[u_L],k\langle X \rangle_j^{< u_R} ] \\
&& +  \sum_{ i, j  ~>~ 0  \atop i+j = \deg (u_L)} k\langle X \rangle_i^{< u_L} \cdot [u_R] \otimes k\langle X\rangle_j^{< u_L} +  \sum_{ i, j  ~>~ 0  \atop i+j = \deg (u_L)} k\langle X \rangle_i^{< u_L}  \otimes k\langle X\rangle_j^{< u_L} \cdot [u_R] \\
&& -  \sum_{ i, j  ~>~ 0  \atop i+j = \deg (u_L)} [u_R]\cdot k\langle X \rangle_i^{< u_L}  \otimes k\langle X\rangle_j^{< u_L} -  \sum_{ i, j  ~>~ 0  \atop i+j = \deg (u_L)} k\langle X \rangle_i^{< u_L}  \otimes [u_R]\cdot k\langle X\rangle_j^{< u_L}  \\
&& +   \sum_{ i, j, i', j'  ~>~ 0  \atop i+j + i'+j' =\deg(u)} k\langle X \rangle_i^{< u_L} \cdot k\langle X \rangle_{i'}^{< u_R}   \otimes k\langle X\rangle_j^{< u_L} \cdot k\langle X \rangle_{j'}^{< u_R} \\
&& - \sum_{ i, j, i', j'  ~>~ 0  \atop i+j + i'+j' =\deg(u)} k\langle X \rangle_i^{< u_R} \cdot k\langle X \rangle_{i'}^{< u_L}   \otimes k\langle X\rangle_j^{< u_R} \cdot k\langle X \rangle_{j'}^{< u_L}  \\
&& + ~ \Big (k\langle X\rangle \otimes k\langle X\rangle \Big )_{< \deg(u)} \\
&\subseteq& 1\otimes [u] + [u] \otimes 1   +    \sum_{ i, j  ~>~ 0  \atop i+j = \deg (u)} k\langle X \rangle_i^{< u}  \otimes k\langle X\rangle_j^{< u} + \big (k\langle X\rangle \otimes k\langle X\rangle \big )_{< \deg(u)}.
\end{eqnarray*}
Here, the inclusion comes from the observations
$$
k\langle X\rangle_{\leq \deg(u_L)}^{< u_L} \subseteq k\langle X\rangle^{< u}, \quad k\langle X\rangle^{< u_R} \subseteq k\langle X\rangle^{< u} \quad  \text{and} \quad [[u_L], k\langle X\rangle^{< u_R}]\subseteq k\langle X\rangle^{< u},
$$
which are followed from  Lemma \ref{lex-order},  Proposition \ref{fact-Lyndon} (L1) and Lemma \ref{commutator-subalgebra} respectively.

Now assume $n>1$.  By induction,
\begin{eqnarray*}
\Delta([u]^n) &=& \Delta([u]) \cdot \Delta ([u]^{n-1}) \\
& \in&   \sum_{p=0}^n \binom{n}{p} ~  [u]^p \otimes [u]^{n-p}    +   \sum_{r, s ~ \geq~ 0 \atop r+s ~<~ n } \sum_{ i, j  ~>~ 0  \atop i+j = (n-r-s)\deg (u)} k\langle X \rangle_i^{< u}\cdot  [u]^{r}   \otimes k\langle X\rangle_j^{< u} \cdot [u]^s\\
&& + \sum_{r, s ~ \geq~ 0 \atop r+s ~<~ n-1 } \sum_{ i, j  ~>~ 0  \atop i+j = (n-1-r-s)\deg (u)} [[u], k\langle X \rangle_i^{< u} ] \cdot  [u]^{r}   \otimes k\langle X\rangle_j^{< u} \cdot [u]^s \\
&&  + \sum_{r, s ~ \geq~ 0 \atop r+s ~<~ n-1 } \sum_{ i, j  ~>~ 0  \atop i+j = (n-1-r-s)\deg (u)} k\langle X \rangle_i^{< u}  \cdot  [u]^{r}   \otimes [[u], k\langle X\rangle_j^{< u}] \cdot [u]^s \\
&& + \big (k\langle X\rangle \otimes k\langle X\rangle \big )_{<\deg(u^n)} \\
&\subseteq&  \sum_{p=0}^n \binom{n}{p} ~  [u]^p \otimes [u]^{n-p}   ~ +~  \sum_{r, s ~ \geq~ 0 \atop r+s ~<~ n } \sum_{ i, j  ~>~ 0  \atop i+j = (n-r-s)\deg (u)} k\langle X \rangle_i^{< u}\cdot  [u]^r   \otimes k\langle X\rangle_j^{< u} \cdot [u]^s \\
&& + \Big (k\langle X\rangle \otimes k\langle X\rangle \Big )_{<\deg(u^n)}.
\end{eqnarray*}
Here, the containment is by the observation $[u]\cdot k\langle X\rangle^{<u} \subseteq [[u], k\langle X\rangle^{<u}] +k\langle X\rangle^{<u} \cdot [u]$; and the inclusion used the observation
$[[u], k\langle X\rangle^{<u}] \subseteq k\langle X\rangle^{<u}$ by Lemma \ref{commutator-subalgebra}.
\end{proof}

The next proposition is the key observation to prove the main results and has an interest in its own right. The following notations are employed. For an ideal $I$ of $k\langle X\rangle$ and a word $w$, we write
\[
k\langle X|I\rangle  \quad \Big(\text{ resp. }  k\langle X|I\rangle^{<w}, \text{ resp. } ~ k\langle X|I\rangle^{\leq w} ~\Big)
\]
for the subalgebra of $k\langle X\rangle$ generated by
$$\{~ [u]~|~ u\in \N_I ~\} \quad \Big ( \text{ resp. } \{~ [u]~|~ u\in \N_I, ~ u<_{\lex}  w ~\}, ~ \text{ resp. } \{~ [u]~|~ u\in \N_I, ~ u \leq_{\lex} w ~\}  ~ \Big).$$
Note that they are all homogeneous subalgebras of $k\langle X\rangle$. We use the notations $k\langle X|I\rangle_n$ and $k\langle X|I\rangle_{\leq n}$ in the obvious sense. Similar convention applies to $k\langle X|I\rangle^{<w}$  and $k\langle X|I\rangle^{\leq w}$.

\begin{proposition}\label{quasi-Lie-imply-LK}
Let $I$ be an ideal of $k\langle X\rangle$ and $A:=k\langle X\rangle/ I$. Assume that there exists a  triangular comultiplication  $\Delta$ on $k\langle X\rangle$ such that $\Delta(I) \subseteq k\langle X\rangle \otimes I + I \otimes k\langle X\rangle$.
\begin{enumerate}
\item For every $I$-reducible Lyndon word $v$,
\[
[v] \in  k\langle X|I\rangle^{<v}_{\deg(v)} + k\langle X|I\rangle_{< \deg(v)} +I.
\]
\item  For every pair of  $I$-irreducible Lyndon words  $u,v\in \N_I$ with $u>_{\lex} v$,
\[
[u][v] -[v] [u] \in k\langle X|I\rangle_{\deg(uv)}^{\leq uv} + k\langle X|I\rangle_{< \deg(uv)} +I.
\]
\item Assume in addition that $k$ is of characteristic $0$. Then the set $\{~[w]+I~|~ w\in \B_I~\}$ form a basis of $A$. Consequently, $\B_I$ equals the set of $I$-irreducible words.
\end{enumerate}
\end{proposition}

The  proof of the above proposition is complicated. To make the reading more fluent, it will be addressed in Section \ref{section-quasi-Lie-imply-LK} through a number of technical lemmas.

\begin{remark}
Ideals of $k\langle X\rangle$ satisfying  the conditions in Proposition \ref{quasi-Lie-imply-LK} are studied extensively in \cite{ZSL3}. It turns out that the quotient algebras of $k\langle X\rangle$ defined by such ideals share many fundamental ring-theoretic and  homological properties.
\end{remark}

To understand Proposition \ref{quasi-Lie-imply-LK} better, we reclaim the following well-known fact  as its consequence. Let ${\rm Lie}(X)$ be the free Lie algebra on $X$. It is the subspace of $k\langle X\rangle$ spanned by the standard bracketing of Lyndon words. Elements of ${\rm Lie}(X)$ are called \emph{Lie polynomials}. Given a set $F$ of Lie polynomials, there is a Hopf algebra $H:=k\langle X\rangle/ (F)$ with comultiplication $\Delta_H$ the map induced by $\Delta_s$ and a Lie algebra $\mathfrak{g}:={\rm Lie} (X)/(F)_L$, where $(F)_L$ denote the Lie ideal of ${\rm Lie} (X)$ generated by $F$. It is well-known that $H$ is isomorphic as a Hopf algebra to the universal enveloping algebra of $\mathfrak{g}$. We refer \cite[p.168, Theorem 7]{J} for the special case where $F$ is the empty set, which is attributed to E. Witt by N. Jacobson; the general case follows easily.

\begin{proposition}\label{Lie-ideal}
Let $I$ be an ideal of $k\langle X\rangle$.
\begin{enumerate}
\item[(1)] If $I$ is generated by Lie polynomials then $\Delta_s(I) \subseteq I\otimes k\langle X\rangle +  k\langle X\rangle \otimes I$.
\item[(2)]  The converse of (1) holds, provided that $k$ is of characteristic $0$.
\end{enumerate}
\end{proposition}

\begin{proof}
Part (1) is by the formula (\ref{ST}) in Remark \ref{remark-standard-comul}.

Next we show Part (2).  Assume $k$ is of characteristic $0$ and $\Delta_s(I) \subseteq I\otimes k\langle X\rangle +  k\langle X\rangle \otimes I$. By Proposition \ref{quasi-Lie-imply-LK} (3), for any nondecreasing nonempty sequence $u_1\leq_{\lex} u_2 \leq_{\lex} \cdots \leq_{\lex} u_n$ in $\N_I$, one may define a functional $\tau_{u_1,\cdots, u_n} : k\langle X\rangle \otimes k\langle X\rangle \to k $ as follows.
\begin{itemize}
\item $\tau_{u_1,\cdots, u_n}(I\otimes k\langle X\rangle +k\langle X\rangle \otimes I) =0$;
\item $\tau_{u_1,\cdots, u_n}([w]\otimes [w']) =0$ for $(w,w')\in (\B_I\times \B_I)  \backslash \{ ([u_1], [u_2]\cdots [u_n])\}$;
\item $\tau_{u_1,\cdots, u_n} ([u_1]\otimes [u_2]\cdots [u_n]) =1$.
\end{itemize}
One may define the functional  $\tau_{\emptyset}$ as well. Now let $v$ be an arbitrary $I$-reducible Lyndon word. Applying Proposition \ref{quasi-Lie-imply-LK} (3), there is a (unique) polynomial $g_v\in I$ of the following form
\[
g_v= [v] +\sum_{w_1,\cdots, w_n \in \N_I \atop w_1\leq_{\lex } \cdots \leq_{\lex} w_n} \lambda_{w_1,\cdots, w_n} [w_1][w_2]\cdots [w_n].
\]
By the formula (\ref{ST}), one has
\[
0=  (\tau_{u_1,\cdots, u_n}\circ \Delta_s)(g_v)  =  p\cdot \lambda_{u_1,\cdots, u_n}, \quad n\geq 2,
\]
where $p$ is the number of $i$ in $\{1,\cdots, n\}$ such that $u_i = u_1$; and
note that
\[
 [v]=g_v- \sum_{w_1,\cdots, w_n \in \N_I,\, n\geq 1 \atop w_1\leq_{\lex } \cdots \leq_{\lex} w_n} \lambda_{w_1,\cdots, w_n} [w_1][w_2]\cdots [w_n]-\lambda_{\emptyset},
\]
we have
\[
0 = (\tau_{\emptyset} \circ \Delta_s ) (g_v) =  \tau_\emptyset (1\otimes [v] + [v]\otimes 1) +\lambda_\emptyset = -\lambda_\emptyset.
\]
Thus $g_v$ is a Lie polynomial with leading word $v$ for any $I$-reducible Lyndon word. By standard  Gr\"{o}bner basis theory (see \cite{Mora}), $I$ is generated by $\{~g_v~ |~ v  \text{ is an $I$-reducible Lyndon word} ~ \}$.
\end{proof}

\section{The structure of connected (graded) Hopf algebras}
\label{section-Hopf-algebra}

In this section we study the structure of connected (graded) Hopf algebras over a field of characteristic $0$.  In addition to the main theorems stated in the introduction, we also observe  some keystone facts about connected Hopf algebras over a field of characteristic $0$.

Let us begin by recalling some notations and definitions on Hopf algebras.  For a general Hopf algebra $H$, the usual notations $\Delta_H$, $\varepsilon_H$ and $S_H$ are employed to denote  the comultiplication,  counit  and  antipode of $H$ respectively. The \emph{coradical} of $H$ is defined to be the sum of all simple subcoalgebras of $H$. It is denoted by $H_{(0)}$. Also,  the \emph{coradical filtration} of $H$ (\cite[Section 5.2]{Mont}) is denoted by $\{H_{(n)}\}_{n\geq0}$.  Note that the notations for coradical and coradical filtration we used  differ from that of \cite{Mont}.
A Hopf algebra is called \emph{connected} if  its coradical  is  one-dimensional.

By a \emph{graded Hopf algebra} we mean a Hopf algebra $H$ equipped with a grading $H=\bigoplus_{n\geq 0}H_n$  such that $H$ is both a graded algebra and a graded coalgebra, and the antipode  preserves the given grading \cite[p.237, Definitions]{Sw}. Such a graded Hopf algebra is called \emph{connected} if $H_0=k$.
Clearly connected graded Hopf algebras are connected Hopf algebras. For a connected Hopf algebra $H$, the coradical filtration $\{H_{(n)}\}_{n\geq 0}$ is a Hopf algebra filtration \cite[p. 62]{Mont}, whence
the associated graded space $\gr_c(H):= \bigoplus_{n\geq 0} H_{(n)}/H_{(n-1)}$ is a connected graded Hopf algebra in the natural way. Moreover, it is easy to show that, when $H$ is a connected graded Hopf algebra, one has $\ker \varepsilon_H= \oplus_{n\geq 1} H_n$; and $$\Delta_H(x)\in x\otimes 1 + 1\otimes x + \bigoplus_{i=1}^{n-1}H_i\otimes H_{n-i}$$for an element $x$ of $H_n$, so that $H_n\subseteq H_{(n)}$ for all $n\geq 0$.

\begin{proposition}\label{connected-field-extension}
Let $H$ be a Hopf algebra over a field $k$. Let $k'$ be a field extension of $k$.
\begin{enumerate}
\item If $H$ is connected then $H\otimes k'$ is connected as a Hopf algebra over $k'$.
\item The converse of (1) holds when $k$ is perfect.
\end{enumerate}
\end{proposition}

\begin{proof}
 (1) Note that $\{H_{(n)}\otimes k'\}_{n\geq 0}$ is a coalgebra  filtration of $H\otimes k'$. By \cite[Lemma 5.3.4]{Mont}, the result follows immediately.

 (2) Now assume $k$ is perfect and $H':=H\otimes k'$ is connected as a Hopf algebra over $k'$. Let $C$ be any simple subcoalgebra of $H$. So $C$ is finite dimensional and $C':=C\otimes k'$ is a finite dimensional subcoalgebra of $H'$. Since $H'$ is connected, one gets $C_{(0)}'=k'$. It follows by \cite[Lemma 5.1.4]{Mont} that $C^*=\Hom_k(C,k)$ is a finite dimensional simple algebra over $k$ and $C^*\otimes k' \cong \Hom_{k'}(C',k')$ is a finite dimensional local algebra  over $k'$ with maximal ideal $\mathfrak{m}:=(C'_{(0)})^{\perp}$, which is nilpotent and of codimension one.
Now consider the composition of the natural $k$-algebra homomorphisms $$C^*\to C^*\otimes k' \to (C^*\otimes k')/\mathfrak{m} \cong k'.$$ Since $C^*$ is simple, this map is injective. So $C^*$ is a commutative simple algebra over $k$.  Therefore, $C^*$ is a field extension of $k$, which is of course seperable because $k$ is perfect. By the theorem and the second corollary in \cite[Section 6.2]{Wat}, $C^*\otimes k'$ is reduced and hence $\mathfrak{m}=0$. So $\dim_{k'} C' =\dim_{k'} C^*\otimes k' =1$, and therefore $C'=k'$. Consequently, $C=k$. Thus, $H$ is connected.
\end{proof}

The following theorem (i.e. Theorem \ref{main-structure-PBW-generator}) is the main result of the paper. As a matter of fact, the definitions and results developed in the last two sections are designed to prove it.

\begin{theorem}\label{structure-PBW-generator}
Assume that the base field $k$ is of characteristic $0$. Let $H$ be a connected graded Hopf algebra. Then there exists an indexed family $\{z_\gamma\}_{\gamma\in \Gamma}$ of homogeneous elements of $H$ of positive degree and a total order $\leq$ on  $\Gamma$ satisfying the following conditions:
\begin{enumerate}
\item for every index  $\gamma \in \Gamma$,
\[
\Delta_H(z_\gamma) \in 1\otimes z_\gamma +z_\gamma\otimes 1 + \bigoplus_{i,j>0,~ i+j = \deg(z_\gamma)} (H^{<\gamma})_i \otimes (H^{<\gamma})_j,
\]
where $H^{<\gamma}$ denotes the subalgebra of $H$ generated by $\{~ z_\delta~|~ \delta \in \Gamma,~ \delta<\gamma ~\}$;
\item  for every pair of indexes
$\gamma, \delta \in \Gamma$ with $\delta<\gamma$,
$$z_\gamma z_\delta -z_\delta z_\gamma \in (H^{<\gamma})_{\deg(z_\gamma z_\delta)};$$
\item the set $\{~ z_{\gamma_1} \cdots z_{\gamma_n} ~|~ n\geq0,~ \gamma_1,\cdots, \gamma_n \in \Gamma, ~ \gamma_1\leq \cdots \leq \gamma_n ~\}$ is a basis of $H$.
\end{enumerate}
Moreover, $\gkdim H$ is finite if and only if $\Gamma$ is finite; and in this case $\gkdim H = \#(\Gamma)$.
\end{theorem}

\begin{proof}
By choosing a set of homogeneous generators of positive degrees, we may fix an alphabet $X$ of graded variables  and consider $H$ as a graded quotient algebra  $k\langle X\rangle/I$ for some homogeneous ideal $I$. For any $f\in k\langle X\rangle$, write $\overline{f} $ for the coset of $f$ in $H$.  By the counitality, one has
\[
\Delta_H(\overline{x}) \in 1\otimes \overline{x} +\overline{x} \otimes 1 + \sum_{i,j>0 ~ i+j =\deg(x)}H_i \otimes H_j, \quad x\in X.
\]
It follows that one may lift  $\Delta_H$  to  a  graded algebra homomorphism $\Delta:k\langle X\rangle\to k\langle X\rangle \otimes k\langle X\rangle$ which satisfies the following condition:
\begin{equation}
\Delta(x) \in 1\otimes x + x \otimes 1 + \sum_{i,j>0 ~ i+j =\deg(x)}k\langle X\rangle_i \otimes k\langle X\rangle_j, \quad x\in X. \tag{${\rm TC}$}
\end{equation}
Note that  $\Delta(I) \subseteq I\otimes k\langle X\rangle +k\langle X\rangle \otimes I$ by the construction.
Further, equip  $X$ with a well order as follows.  First fix a well-order $\leq_r$ on $X_r:=\{~x\in X~|~ \deg(x) =r ~\}$ for each integer $r\geq 1$; then for $x_1,x_2\in X$,
\begin{eqnarray*}
\label{definition-deglex}
x_1 \leq x_2\,\, \Longleftrightarrow \,\, \left\{
\begin{array}{llll}
\deg(x_1)<\deg(x_2), \quad \text{or} && \\
\deg(x_1) =\deg(x_2)\, \text{ and } x_1\leq_{\deg(x_1)} x_2.
\end{array}\right.
\end{eqnarray*}
It is easy to read from (TC) that $\Delta$ is a graded triangular comultiplication. Now let $\Gamma:=\N_I$ and
\[
z_\gamma :=[\gamma] +I \in k\langle X\rangle/I = H
\]
for any $\gamma\in \Gamma$. Moreover, let $\leq$ be the restriction of $\leq_{\lex}$ on $\N_I$.
By Proposition \ref{quasi-Lie-imply-LK} (1) one has
$$k\langle X|I\rangle^{<w} +I= k\langle X\rangle^{<w}+I$$
for every Lyndon word $w$.  Part (1) then follows immediately from Proposition \ref{comultiplication}.
Part (2) and Part (3) are direct consequences of Proposition \ref{quasi-Lie-imply-LK} (2) and Proposition \ref{quasi-Lie-imply-LK} (3) respectively.

Next we show the last statement. For any finite subset $\Xi$ of $\Gamma$, let $H_\Xi$ be the  subalgebra of $H$  generated by $\{~z_\xi ~|~ \xi \in \Xi~\}$. For any integer $n$, let $d_\Xi(n)$ be the number of nondecreasing sequences $(\gamma_1,\cdots, \gamma_p)$ in $\Xi$ such that  $\deg(z_{\gamma_1} \cdots z_{\gamma_p})=n$. By Part (3) and a simple combinatorial argument,
\begin{equation*}
 \sum_{n\geq0} \dim(H_{\Xi})_n ~ t^n \geq \sum_{n\geq 0} d_\Xi(n) ~ t^n= \prod_{u \in \Xi} \big (1-t^{\deg (u)}  \big )^{-1},
\end{equation*}
where  the inequality means that the difference of the series  has no negative coefficients. Consequently,
\begin{equation}\label{Hilbert-GKdim}
\gkdim(H_\Xi) = \limsup_{n\to \infty} \log_n \sum_{i\leq n} \dim(H_\Xi)_i \geq \limsup_{n\to \infty} \log_n \sum_{i\leq n} d_\Xi(i)= \#(\Xi). \tag{$*$}
\end{equation}
Here, the first equality is by \cite[Lemma 6.1 (b)]{KL} and the second equality is by \cite[Proposition 2.21]{ATV2}.
It follows that if $\Gamma$ is infinite then $\gkdim(H) \geq \gkdim(H_\Xi) \geq \#(\Xi)$ for any finite subset $\Xi$ of $\Gamma$ and hence  $\gkdim(H) =\infty$. If $\Gamma$ is finite then $H=H_{\Gamma}$ and the inequality in ($*$)  becomes an equality for $\Xi=\Gamma$ by Part (3), whence $\gkdim(H) = \gkdim(H_\Gamma) = \#(\Gamma)$.
\end{proof}

\begin{remark}
It is worth mentioning that for a family $\{z_\gamma\}_{\gamma\in \Gamma}$ of homogeneous elements of $H$ and a total order on the index set $\Gamma$ as in the above theorem,  the function $\Gamma \to \mathbb{N}$ sending $\gamma$ to $\deg(z_\gamma)$ may fail to preserve order, that is, it may occur that $\deg(z_\gamma) >\deg(z_\delta)$ for some indexes $\gamma < \delta$. For example, take $X=\{x<y\}$. It may happen that $\Gamma=\{x, yx,y\}$, and one has $yx<_{\lex} y$ but $\deg(yx)>\deg(y)$.
\end{remark}

\begin{proposition}\label{gkdim-equal}
Assume $k$ is of characteristic $0$. Let $H$ be a  connected Hopf algebra. Then
$$\gkdim H = \gkdim \gr_c(H) \in \mathbb{N} \cup\{\infty\}.$$
\end{proposition}

\begin{proof}
If $\gkdim \gr_c(H)$ is infinite then $\gkdim(H) \geq \gkdim \gr_c(H) =\infty$ by \cite[Lemma 6.5]{KL}, and so $\gkdim(H) =\infty$. If $\gkdim \gr_c(H)$ is finite then $\gr_c (H)$ is finitely generated and $\gkdim \gr_c(H) \in \mathbb{N}$ by Theorem \ref{structure-PBW-generator}, whence $\gkdim H= \gkdim \gr_c(H) \in \mathbb{N}$ by \cite[Proposition 6.6]{KL}.
\end{proof}

The following result is due to Brown, Gilmartin and Zhang (\cite[Theorem 2.4 (1, 5)]{BGZ}). We provide a slightly different proof without the use of Cartier-Kostant Theorem on the structure of cocomummutative connected Hopf algebras over a field of characteristic $0$ (\cite[Theorem 5.6.5]{Mont}).

\begin{proposition} (\cite[Theorem 2.4 (1, 5)]{BGZ})
Assume that $k$ is of characteristic $0$.  Let $H$ be a Hopf algebra that is connected graded and locally finite as an algebra. Here, locally finite  means that each component $H_n$ is finite dimensional.
Then there is a unique sequence of natural numbers $n_1,n_2,\cdots$ such that the Hilbert series of $H$ is $$\prod\nolimits_{i=1}^\infty (1-t^i)^{-n_i}.$$ Moreover, $\gkdim(H) =\sum_{i=1}^\infty n_i$, which is either infinite or a natural number.
\end{proposition}

\begin{proof}
Let $\mathfrak{m}:= \sum_{n\geq 1} H_n$. By \cite[Lemma 2.1 (2)]{BGZ} we may assume $\mathfrak{m} =\ker (\varepsilon_H)$. By \cite[Lemma 3.3]{GZ}, it is easy to see that $B:= \bigoplus_{i,j\geq 0} (\mathfrak{m}^i/\mathfrak{m}^{i+1})_j$ is a $\mathbb{Z}^2$-graded Hopf algebra with comultiplication and counit  defined in \cite[Lemma 3.2]{GZ}. So $B=\bigoplus_{j\geq 0} B_j$ is a graded Hopf algebra with $B_j=\bigoplus_{i\geq 0}  (\mathfrak{m}^i/\mathfrak{m}^{i+1})_j$.  Clearly, $B$ and $H$ have the same Hilbert series. In particular, $B_0=k$, so $B$ is a connected graded Hopf algebra.

Moreover, we claim that $\gkdim H=\gkdim B$. Indeed, if $\gkdim B=\infty$ then $\gkdim H=\infty$ by \cite[Lemma 6.5]{KL}. Suppose  $\gkdim B<\infty$. By Theorem \ref{structure-PBW-generator}, $B$ is affine, whence $H$ is affine too by a standard application of filtered-graded methods because $(\mathfrak{m}^i)_r=0$ for integers $0\leq r<i$. Then by \cite[Lemma 6.1 (b)]{KL}, $\gkdim B$ and $\gkdim H$ are both determined by their Hilbert series which are same.  Thus $\gkdim B=\gkdim H$.

According to the above discussion, to see the result one may assume that $H$ itself is a connected graded Hopf algebra.
The uniqueness of such sequence is clearly determined  by the Hilbert series of $H$. To see the existence, choose an indexed family $\{z_\gamma\}_{\gamma\in \Gamma}$ of homogeneous elements of $H$ of positive degrees and a total order $\leq$ on $\Gamma$ as in Theorem \ref{structure-PBW-generator}. Then the set  $$\Gamma_i:=\{~\gamma\in \Gamma ~|~ \deg(z_\gamma) =i~ \}$$ is finite for every integer $i\geq 1$, since $H$ is locally finite. Set $n_i$ to be the cardinality of $\Gamma_i$ for every $i\geq 1$.  Applying Theorem \ref{structure-PBW-generator} (3), a simple combinatorial argument tells us that the Hilbert series of $H$ is of the required form. The last statement is clear because $\gkdim H= \#(\Gamma)$.
\end{proof}

\begin{proposition}\label{connected-commutative-Hopf-algebra}
Assume that  $k$ is of characteristic $0$. Let $H$ be a connected commutative Hopf algebra. Then  $H$ is isomorphic as an algebra to the polynomial algebra in some family of variables.
\end{proposition}

\begin{proof}
Choose an indexed family $\{z_\gamma\}_{\gamma\in \Gamma}$ of homogeneous elements of $\gr_c (H)$ of positive degrees and   a total order $\leq$ on $\Gamma$ as in Theorem \ref{structure-PBW-generator}. For each index $\gamma$ with $\deg(z_\gamma)=n$,  pick an element $a_\gamma\in H_{(n)}$ such that $a_\gamma+H_{(n-1)} =z_\gamma$ in $\gr_c(H)$.  By a standard application of filtered-graded methods,
$$\{~ a_{\gamma_1} \cdots a_{\gamma_n} ~|~ n\geq0,~ \gamma_1,\cdots, \gamma_n \in \Gamma, ~ \gamma_1\leq \cdots \leq \gamma_n ~\}$$ is a basis of $H$. Since $H$ is commutative, $H$ is obviously a polynomial algebra.
\end{proof}

\begin{proposition}\label{graded-commutative-Hopf-algebra}
Assume that $k$ is of characteristic $0$. Let $H$ be a commutative Hopf algebra which is connected graded as an algebra. Then $H$ is isomorphic as a graded algebra to the graded polynomial algebra in some family of graded variables.
\end{proposition}

\begin{proof}
Let $\mathfrak{m}:= \sum_{n\geq 1} H_n$.  By the argument of \cite[Lemma 2.1 (2)]{BGZ}, to see the result we may assume $\mathfrak{m} =\ker (\varepsilon_H)$. By \cite[Lemma 3.3]{GZ}, $B= \bigoplus_{i\geq 0} B_i $ with $B_i:=\mathfrak{m}^i/\mathfrak{m}^{i+1}$ is a connected graded Hopf algebra with comultiplication and counit  defined in \cite[Lemma 3.2]{GZ}.  Then by Theorem \ref{structure-PBW-generator} (3),
$B$ is freely generated as a commutative algebra on any choice of bases for $B_1$, since $B$ is generated by $B_1$ and commutative. Note that $B$ is actually $\mathbb{Z}^2$-graded as an algebra with $B_{(i,j)} = (\mathfrak{m}^i/\mathfrak{m}^{i+1})_j.$ Choose an indexed family of homogeneous elements $\{a_\gamma\}_{\gamma\in \Gamma}$ of $\mathfrak{m}$ such that the family of cosets  $\{~ \overline{a_\gamma} ~ \}_{\gamma\in \Gamma}$, where $\overline{a_\gamma}:=a_\gamma+\mathfrak{m}^2$, forms a  basis of $B_1$. Fix a total order $\leq$ on $\Gamma$. Then the set
$$\{~ \overline{a_{\gamma_1}} \cdots \overline{a_{\gamma_n}} ~|~ n\geq0,~ \gamma_1,\cdots, \gamma_n \in \Gamma, ~ \gamma_1\leq \cdots \leq \gamma_n ~\}$$
forms a $\mathbb{Z}^2$-homogeneous basis for $B$. By a standard application of filtered-graded method and the observation that $(\mathfrak{m}^i)_r = 0$ for integers $0\leq r<i$, it is not hard to conclude that
the set $$\{~ a_{\gamma_1} \cdots a_{\gamma_n} ~|~ n\geq0,~ \gamma_1,\cdots, \gamma_n \in \Gamma, ~ \gamma_1\leq \cdots \leq \gamma_n ~\}$$ forms a homogeneous basis for $H$. Since $H$ is commutative,  $H$ is a graded  polynomial algebra.
\end{proof}

\begin{remark}\label{remark-commutative}
By Proposition \ref{connected-field-extension}, the assumption that the base field is algebraically closed in \cite[Theorem 0.1]{BGZ} can be omitted. That is to say among others that  if the base field is of characteristic $0$, then the following conditions for an affine commutative Hopf algebra $H$ are equivalent:
\begin{enumerate}
\item $H$ is a connected Hopf algebra.
\item $H$ is connected graded as an algebra.
\item $H$ is isomorphic as an algebra to a polynomial algebra.
\end{enumerate}
By Proposition \ref{connected-commutative-Hopf-algebra} and Proposition \ref{graded-commutative-Hopf-algebra}, we have $(1)\Rightarrow (2) \Leftrightarrow (3)$ for $H$ without being affine. But does the implication $(2)\Rightarrow (1)$ hold for $H$ without being affine?
\end{remark}

\begin{remark}\label{remark-finiteness}
Let $H$ be a connected Hopf algebra over a field of characteristic $0$. By \cite[Proposition 6.4]{Zh},  $\gr_c(H)$ is a commutative connected graded Hopf algebra. So one can easily  deduce from Proposition \ref{gkdim-equal} and Proposition \ref{graded-commutative-Hopf-algebra} that the following finiteness conditions are equivalent: (1) $H$ has finite GK dimension; (2) $\gr_c(H)$ has finite GK dimension; (3) $\gr_c(H)$ is affine; (4) $\gr_c(H)$ is noetherian. Moreover, if these equivalent conditions hold, then $H$ is affine and noetherian. See \cite[Theorem 6.9]{Zh} when the base field  $k$ is algebraically closed. It is an open question that whether  $H$ is noetherian can imply that $H$ has finite GK dimension (\cite[Question K]{BG0}).
\end{remark}

Let $H$ be a connected Hopf algebra of finite GK dimension over a field of characteristic $0$. Then the primitive space of $H$ is of finite dimension. Note that a graded Hopf algebra $K=\oplus_{n\geq 0}K_n$ is coradically graded (i.e. $K_{(0)} = K_0 $ and $K_{(1)} = K_0+K_1$) if and only if the iterated comultiplications $\Delta^{(n)}_K:K\to K^{\otimes n}$ restrict to an injective map  $K_n \to K_1^{\otimes n}$ for all integers $n\geq 1$ (see the proof of \cite[Lemma 5.5]{AnSc} for this fact). Since $\gr_c(H)$ is coradically graded (\cite[Proposition 7.9.4]{Rad}),
$\gr_c(H)$ is locally finite and whence the graded dual $\gr_c(H)^*$ is too. The primitive space of $\gr_c(H)^*$, which is  denoted by $\mathfrak{L}(H)$, is called the \emph{lantern} of $H$ (\cite[Definition 5.4]{BG}).  Note that $\mathfrak{L}(H)$ is a positively graded Lie algebra.

The  following observation on $\mathfrak{L}(H)$ is owed to Wang, Zhang and Zhuang (\cite[Lemma 1.3]{WZZ}) under the assumption that the base field $k$ is algebraically closed of characteristic $0$. We slightly modify their argument to omit the assumption of algebraically closedness on the base field.

\begin{proposition}\label{lattern}
Assume that  $k$ is of characteristic $0$. Let $H$ be a  connected Hopf algebra of finite GK dimension $d$. Then $\mathfrak{L}(H)$ is of dimension $d$ and generated in degree one.
\end{proposition}

\begin{proof}
Note that $K:=\gr_c(H)$ is coradically graded (\cite[Proposition 7.9.4]{Rad}). Then by \cite[Lemma 5.5]{AnSc}, the graded dual $K^*$ is generated in degree $1$ as an algebra. Note that
\[
(K^*)_1 = \mathfrak{L}(H)_1 \subseteq U(\mathfrak{g}) \subseteq U(\mathfrak{L}(H)) \subseteq K^*,
\]
where $\mathfrak{g}$ is the Lie subalgebra of $\mathfrak{L}(H)$ generated by $\mathfrak{L}(H)_1$.  So one has $U(\mathfrak{g}) = U(\mathfrak{L}(H)) =K^*$, and whence $\mathfrak{g} =\mathfrak{L}(H)$. Therefore $\mathfrak{L}(H)$ is generated in degree $1$.  Finally,
\[
\dim \mathfrak{L}(H) =  \gkdim K^* =\gkdim K = \gkdim H=d,
\]
where the first equality is well-known (see \cite[Example 6.9]{KL}),  the second equality is by  \cite[Lemma 6.1 (b)]{KL} and the observations that   $K$ and $K^*$ are both finitely generated and have the same Hilbert series,  and where the third equality is by Proposition \ref{gkdim-equal}.
\end{proof}

Connected Hopf algebras of GK dimension $\leq 4$ over algebraically closed fields of characteristic $0$ were  classified by Zhuang \cite{Zh} and Wang, Zhang and Zhuang \cite{WZZ}. The next result may shed some light on the classification of connected Hopf algebras of higher finite GK dimension. For a Hopf algebra $H$ and an element $a\in H$ we denote by $\rho_H(a)$ the least number $n\geq 0$ such that $a\in H_{(n)}$.

\begin{theorem}\label{structure-connected-Hopf}
Assume that  $k$ is of characteristic $0$. Let $H$ be a  connected Hopf algebra of finite GK dimension $d$. Then there is a sequence of elements $a_1,\cdots, a_d \in \ker(\varepsilon_H)$ such that
\begin{enumerate}
\item $1\leq \rho_H(a_1) \leq \rho_H(a_2) \leq \cdots \leq \rho_H(a_d)$;
\item for every integer $r\geq1$, the number of indexes $i$ with $\rho_H(a_i) =r$ is $\dim \mathfrak{L}(H)_r$;
\item for every non-negative integer $n$, the space $H_{(n)}$ has a basis
$$
\{~ a_{1}^{r_1} \cdots a_{d}^{r_d} ~|~ r_1,\cdots, r_d\geq0, ~ \rho_H(a_1)\cdot r_1 +\cdots +\rho_H(a_{d})\cdot r_d \leq n ~\};
$$
\item for every pair of indexes $1\leq i<j\leq d$,
$$a_j a_i -a_i a_j \in H_{(\rho_H(a_i) +\rho_H(a_j)-1)};$$
\item  for every index  $r=1,\cdots, d$,
\[
\Delta_H(a_r) \in 1\otimes a_r +a_r\otimes 1 + \sum_{i,j>0, ~ i+j = \rho_H(a_r)} H_{(i)} \otimes H_{(j)}.
\]
\end{enumerate}
\end{theorem}

\begin{proof}
By \cite[Proposition 6.4]{Zh} and Proposition \ref{gkdim-equal}, $\gr_c (H)$ is commutative and of GK dimension $d$.  So by Theorem \ref{structure-PBW-generator},  one may choose a sequence of homogeneous elements $z_1,\cdots,
z_d \in \gr_c(H)$ of positive degrees such that $\deg(z_1)  \leq \cdots \leq \deg(z_d)$ and
\[
\{~ z_{1}^{r_1} \cdots z_{d}^{r_d} ~|~ r_1,\cdots,
r_d\geq0 ~\}
\]
forms a basis of $\gr_c(H)$. For each index $i=1,\cdots,d$,  pick an element $a_i\in H_{(\deg(z_i))}$ such that $$a_i+H_{(\deg(z_i)-1)} =z_i \quad \text{in} \quad \gr_c (H).$$ By replacing $a_i$ by $a_i-\varepsilon_H(a_i)$, one may assume $a_i\in \ker(\varepsilon_H)$.
By the construction,
$\rho_H(a_i) =\deg(z_i)$, so Part (1)
follows.  By a standard application of filtered-graded methods, we have Part (3). Since $\gr_c(H)$ is commutative, we have Part (4). By the argument of \cite[Lemma 5.3.2 (2)]{Mont}, we have Part (5).

To see Part (2), it is equivalent to show for every integer $r\geq 1$ the number of indexes
$i$ with $\deg(z_i) = r$ is $\dim \mathfrak{L}(H)_r$. For any (commutative) monomial $u= z_{1}^{r_1} \cdots z_{d}^{r_d}$, let $u^*\in \gr_c(H)^*$ be the linear functional that sends $u$ to $1$ and  all
other monomials to $0$. Then these linear functionals, with $u$ running over all monomials in $z_1,\cdots, z_d$, form a homogeneous basis of  $\gr_c(H)^*$ with $\deg(u^*) =\deg(u)$.  By Proposition \ref{lattern}, to prove (2)  it suffices to show $z_i^*$ is primitive for $i=1,\cdots, d$. Clearly,
\[
\Delta_{\gr_c(H)^*}(z_i^*) (u\otimes v) = 0
\]
for monomials $u,v$ with $(u,v) \not\in \{(1,z_i), ~ (z_i,1)\}$.
It follows that
\[
\Delta_{\gr_c(H)^*}(z_i^*) = \mu_i \otimes z_i^* +z_i^*\otimes  \nu_i , \quad \mu_i, ~ \nu_i\in k.
\]
Since $\varepsilon_{\gr_c(H)^*} (z_i^*) =0$,  by counitality one has $\mu_i=\nu_i=1$. Thus $z_i^*$ is indeed primitive.
\end{proof}

The condition that the tensor algebra $H\otimes H$ is a domain is important in the study of Hopf Ore extensions by Brown, O'Hagan, Zhang and Zhuang \cite{BOZZ}. At the beginning of the Subsection 2.8 of \cite{BOZZ}, the authors indicate that if $H$ is a connected Hopf algebra then $H^{\otimes n}$ is a domain for every integer $n\geq 1$, provided that the base field $k$ is algebraically closed of characteristic $0$. We extend their result to arbitrary fields of characteristic $0$ by a different argument. As a consequence, some results of \cite{BOZZ} can be generalized by omitting the  assumption that the base field is  algebraically closed.

\begin{theorem}\label{connected-Hopf-are-domain}
Assume that  $k$ is of characteristic $0$. Let $H$ be a connected Hopf algebra. Then for every algebra $A$ which is a domain, $H\otimes A$ is a domain. In particular, $H^{\otimes n}$ are domains for $n\geq 1$.
\end{theorem}

\begin{proof}
By \cite[Proposition 6.4]{Zh} and Proposition \ref{connected-commutative-Hopf-algebra}, $\gr_c(H)$ is a polynomial algebra  in some family of variables. Consider the natural filtration $\mathcal{F} = \{~ H_{(n)} \otimes A ~\}_{n\geq 0}$ on $H\otimes A$.  It is easy to check that  $$\gr_{\mathcal{F}} (H\otimes A) \cong \gr_c (H) \otimes A.$$
So $\gr_{\mathcal{F}} (H\otimes A)$ is a domain.  Then by \cite[Proposition 1.6.6]{MR},  $H\otimes A$ is a domain.
\end{proof}

Recall from \cite[Definition 3.1]{BOZZ} that a Hopf algebra $H$ is called an \emph{iterated Hopf Ore extension of $k$} (IHOE, for short)  if there is a  chain of Hopf subalgebras $$k= H^0 \subset H^1 \subset \cdots \subset H^n =H$$ with each of the extensions $H^i\subset H^{i+1}$ being  an Ore extension as algebras.  By the argument of \cite[Proposition 2.8]{BOZZ} and the above theorem one immediately has:

\begin{corollary}\label{IHOE-connected}(\cite[Theorem 1.4]{BOZZ})
Assume $k$ is of characteristic $0$. Then IHOEs are all connected Hopf algebras. \hfill $\Box$
\end{corollary}

The following result is the most surprising corollary of Theorem \ref{structure-PBW-generator}. It provides a large class of connected Hopf algebras of finite GK dimension that are IHOEs.

\begin{theorem}
\label{structure-connected-graded-Hopf}
Assume that $k$ is of characteristic $0$. Let $H$ be a connected graded Hopf algebra  of finite GK dimension $d$. Then $H$ is an IHOE. More precisely, there is a finite sequence $z_1,\cdots, z_d $ of homogeneous elements of $H$  of positive degrees such that
$$
H = k[z_1][z_2;\delta_2]\cdots [z_d; \delta_d]
$$
and
$$
\Delta_H(z_r) \in 1\otimes z_r + z_r\otimes 1 + \sum_{i,j>0, ~ i+j =\deg(z_r)}(H^{\leq r-1})_i \otimes (H^{\leq r-1})_j, \quad r=1,\cdots,d,
$$
where $H^{\leq 0}=k$, $H^{\leq i}$ is the subalgebra of $H$ generated by $z_1,\cdots, z_i$ for $i=1,\cdots,d$
and  $\delta_i$ a derivation of $H^{\leq i-1}$ for $i=2,\cdots, d$.
In particular, $H^{\leq 0}, H^{\leq 1}, \cdots, H^{\leq d}$ are all graded Hopf subalgebras of $H$.
\end{theorem}

\begin{proof}
Fix an  indexed family $\{z_\gamma\}_{\gamma\in \Gamma}$ of homogeneous elements of $H$ of positive degrees and a total order $\leq$ on $\Gamma$ as in Theorem \ref{structure-PBW-generator}. Since $H$ is of finite GK dimension $d$, we may assume $\Gamma= \{~1,\cdots, d ~\}$ equipped with the natural order. By Theorem \ref{structure-PBW-generator} (2), it is not hard to check that
$$z_i \cdot f -f\cdot z_i \in (H^{\leq i-1})_{\deg(z_i) +\deg(f)}$$ for every homogeneous element $f$ in $H^{\leq i-1}$ and every $i=1,\cdots,d$.
Then by induction on $i$ and Theorem \ref{structure-PBW-generator} (3), one may readily conclude that the set
\[ \{~ z_1^{r_1}\cdots z_{i}^{r_{i}} ~ |~ r_1,\cdots, r_{i}\geq 0 ~\}\]
is a basis of the subalgebra $H^{\leq i}$. Consequently, $H^{\leq 1} = k[z_1]$  and
$H^{\leq i}$ is a free graded left $H^{\leq i-1}$-module with homogeneous basis $\{~ 1, ~ z_i, ~ z_i^2, \cdots ~\}$ for $i= 2,\cdots, d$. Define a derivation $\delta_i:H^{\leq i-1} \to H^{\leq i-1}$ by $$f\mapsto z_i \cdot f -f\cdot z_i, \quad f\in H^{\leq i-1},$$  one gets
$H^{\leq i} = H^{\leq i-1} [z_i;\delta_i].$
Finally, the required formula for $\Delta_H(z_r)$ is by Theorem \ref{structure-PBW-generator} (1). The last statement is clear from the expansion $\Delta_H(z_r)$ for  $r=1,\cdots, d$.
\end{proof}

Finally, we show that some fundamental homological properties are enjoyed by connected Hopf algebras of finite GK dimension over a field of characteristic $0$. The result is well-known to experts when the base field is algebraically closed. The proof is closely following the idea  of  Zhuang in \cite[Corollary 6.10]{Zh} and of Brown and Gilmartin in \cite[Theorem 4.3]{BG}. The unexplained terminology used in the theorem is standard, and can be found for example in  \cite{Bj, BZ,  Le,  MR,  RRZ, ZSL1,ZSL2}.

\begin{theorem}
Assume that $k$ is of characteristic $0$. Let $H$ be a connected  Hopf algebra of finite GK dimension $d$. Then
\begin{enumerate}
\item $H$ is universally noetherian (i.e., $H\otimes R$ is noetherian for any noetherian algebra $R$).
\item $H$ is Auslander regular, (GK-)Cohen-Macaulay and skew $d$-Calabi-Yau.
\item $H$ is of global dimension $d$ and Krull dimenision $\leq d$.
\item $H$ is Artin-Schelter regular of dimension $d$ (as an augmented algebra by $\varepsilon_H$) .
\end{enumerate}
\end{theorem}

\begin{proof}
By \cite[Proposition 6.4]{Zh}, Proposition \ref{gkdim-equal} and Proposition \ref{graded-commutative-Hopf-algebra},  $\gr_c(H)$ is isomorphic as a graded algebra to the graded polynomial algebra in graded variables $x_1,\cdots, x_d$.
For any algebra $R$,  there is  a natural filtration $\mathcal{F} = \{~ H_{(n)} \otimes R ~\}_{n\geq 0}$ on $H\otimes R$ and it is easy to check that  $$\gr_{\mathcal{F}} (H\otimes R) \cong \gr_c (H) \otimes R \cong R[x_1,\cdots, x_d].$$
Then Part (1) follows by \cite[Theorem 1.6.9]{MR} and \cite[Theorem 1.2.9]{MR}.

It is well-known that a filtered algebra is  Auslander-regular, Cohen-Macaulay and skew $d$-Calabi-Yau if  the associated graded algebra is so, provided that the filtration is positive and each layer is finite dimensional. See \cite[Proposition 7.6]{ZSL1} and \cite[Theorem 4.4]{ZSL2} for references. Since  $k[x_1,\cdots, x_d]$ is Auslander-regular, Cohen-Macaulay and $d$-Calabi-Yau, Part (2) follows.

By \cite[Lemma 6.5.6, Corollary 7.6.18]{MR}, $H$ is of global dimension and Krull dimension at most $d$. Since $H$ is Cohen-Macaulay, by definition one has $$\gkdim ({}_Hk) + \min\{ ~ i ~|~  \text{\rm Ext}^i_H({}_Hk, H) \neq 0 ~\} =\gkdim(H).$$ It follows that  $\text{\rm Ext}^{i}_H({}_Hk, H) =0$ for $i<d$ and  $\text{\rm Ext}^{d}_H({}_Hk, H) \neq 0$. Therefore, the global dimension of $H$ is $d$, and Part (3) is proved.

To see Part (4) it remains to show $\dim \text{\rm Ext}^{d}_H({}_Hk, H) \leq 1.$
By \cite[Proposition 3.1]{Bj}, there is a good right $H$-module filtration  on $\text{\rm Ext}^{d}_H({}_Hk, H)$ such that the associated graded right $\gr_c(H)$-module of $\text{\rm Ext}^{d}_H({}_Hk, H)$ with respect to this filtration  is a subfactor of $\text{\rm Ext}_{\gr_c(H)}^d (k, \gr_c(H))$. It is well-known that $\gr_c(H)=k[x_1,\cdots, x_n]$ is Artin-Schelter regular  of dimension $d$ as a connected graded algebra, and it follows that  $\text{\rm Ext}_{\gr_c(H)}^d (k, \gr_c(H))$ is of dimension one.  Thus $\dim \text{\rm Ext}^{d}_H({}_Hk, H) \leq 1$ as required.
\end{proof}

\section{Proof of Proposition \ref{quasi-Lie-imply-LK}}
\label{section-quasi-Lie-imply-LK}

This section  is devoted to prove Proposition \ref{quasi-Lie-imply-LK}. The argument has its origin in \cite{Kh, Uf}.  The notations and conventions employed in Section \ref{section-proof-structure-PBW-generator} are retained. So $X$ is a well-ordered alphabet and the free algebra $k\langle X\rangle$ is connected graded with letters homogeneous of positive degrees.

\begin{lemma}\label{killing-functional}
Let $I$ be an ideal of $k\langle X\rangle$. For every Lyndon word $u$ and any integer $l>0$,
$$
\sum_{i=0}^{l-1} \big (k\langle X\rangle^{< u} \cdot [u]^i \big )_{\leq \deg(u^l)} \subseteq I + \text{\rm Span of }~  \{ ~ [w] ~ | ~ w \text{ is $I$-irreducible}, ~ w\neq  [u]^l ~\}.
$$
\end{lemma}

\begin{proof}
Firstly we prove that every word $v$ with  $v<_{\glex} u^l$ is contained in the space on the right of the inclusion, by induction on all such words with respect to the graded lex order. Clearly, the empty word is in the space of the right hand. For any word $v<_{\glex} u^l$, there are two situations. If $v$ is $I$-reducible, there exists an element $g_v\in I$ such that $\mathrm{LW}(g_v)=v$ and $v=g_v+(v-g_v)$; If $v$ is $I$-irreducible, then $[v]$ is contained in the space of the right hand and $v=[v]+(v-[v])$. By Lemma \ref{bracketing-leading} and the induction hypothesis, $v$ lies in the space on the right of the inclusion in both situations.

Now let $f$ be an arbitrary nonzero element of the space on the left of the inclusion.  By Lemma \ref{reordering-bracketing},
$$
f  =  \sum  \lambda_{(a_1^{(i)},\cdots, a_p^{(i)})} ~ [a_1^{(i)}]\cdots [a_p^{(i)}] [u]^i, \quad \lambda_{(a_1^{(i)},\cdots, a_p^{(i)})} \in k,
$$
where $i$ runs over integers from $0$ to $l-1$ and $(a_1^{(i)}, \cdots, a_{p}^{(i)} )$ runs over nondecreasing sequence of Lyndon words that $<_{\lex} u$ and of total degree $\leq (l-i) \cdot \deg(u)$. By Lemma \ref{lex-order-Lyndon} and Lemma \ref{bracketing-leading}, the leading word of $f$ is $<_{\glex} u^l$, and so $f$ is contained in the space on the right of the inclusion.
\end{proof}

Let $I$ be a proper ideal of $k\langle X\rangle$.
For a Lyndon  word $u\in \mathbb{L}$, define the \emph{height} of $u$ by
\[
h_I(u):= \min\{ ~ n\geq 1 ~ | ~  u^n  ~ \text{is $I$-reducible} ~ \}.
\]
By the convention,
$h_I(u) =\infty$ if there is no integer $n$ such that $u^n$ is $I$-reducible.  Note that $h_I(u)=1$ if  $u$ is $I$-reducible; and $h_I(u)\geq 2$ for every $I$-irreducible Lyndon word $u$.
Write  $$ \C_I:= \{~ u_1^{e_1} \cdots u_r^{e_r} ~|~ r\geq 0, ~ u_1<_{\lex} \cdots <_{\lex } u_r,  ~ u_i \in \N_I \text{ and } 0 \leq e_i < h_I(u_i) ~\}.$$
Note that $\C_I\subseteq \B_I$ and it contains all $I$-irreducible words. So by Lemma \ref{standard-basis-2},  the residue classes of the standard bracketing of words in $\C_I$ span the quotient algebra $k\langle X\rangle/ I$.

\begin{lemma}\label{quasi-Lie-basis}
Let $I$ be a proper ideal of $k\langle X\rangle$ and $A: = k\langle X\rangle/I$. Assume that there exists a  triangular comultiplication  $\Delta$ on $k\langle X\rangle$ such that $\Delta(I) \subseteq k\langle X\rangle \otimes I + I \otimes k\langle X\rangle$. Then $\{~ [w]+I ~| ~ w\in \C_I ~\}$ is a basis of   $A$. Consequently, $\C_I$ equals the set of $I$-irreducible words.
\end{lemma}

\begin{proof}
It remains to see elements in the set $\{~ [w]+I ~| ~ w\in \C_I ~\}$ are linearly independent. For $n\geq 0$, let  $$\C_I^n := \{~ w\in \C_I  ~| ~  \deg(w) =  n ~\}$$  and $U_{\leq n}$ the subspace of $k\langle X\rangle_{\leq n}$ spanned by $\{~ [w]  ~| ~ w\in  \C_I^{\leq n} ~\}$.  To see the result it suffices to show $U_{\leq n} \cap I = 0$ for each integer $n\geq 0$.  We proceed by induction on $n$. It is clear that $U_{\leq 0}\cap I=0$ because $I$ is proper.

Now assume $n > 0 $. Assume  $ U_{\leq n} \cap I$ has a nonzero element $f$. Then,
\[
f\in  \sum_{w\in \C_I^n} \lambda_w [w] + k\langle X\rangle_{<n}, \quad \lambda_w\in k
\]
with some $\lambda_w$ nonzero. Choose $u$ as the lexicographically biggest Lyndon word occurring in the Lyndon decomposition of words $w\in \C_I^n$ with $\lambda_w\neq 0$. Then $u$ occurs in the Lyndon decompositions of words $w\in \C_I^n$ with $\lambda_w\neq 0$ only at the end.  Let $l$ be the maximal number of occurrences of $u$ in a Lyndon decomposition of word $w\in \C_I^n$ with $\lambda_w\neq 0$. Thus we can express  $f$ in the following way:
\[
f \in  \sum_{i=0}^l  \sum_{w_i\in P_i} \gamma_{w_i} [w_i] [u]^{i} + k\langle X\rangle_{<n},
\]
where $$P_i = \{~w \in \C_I^{n-\deg(u^i)}~|~ \text{Lyndon words in the Lyndon decomposition of $w$ are } <_{\lex} u, ~ \lambda_{wu^i} \neq 0 ~\}$$ and $\gamma_{w_i} := \lambda_{w_iu^{i}} $. Note that
$$P_l \neq \emptyset \quad  \text{and} \quad [w_i] \in k \langle X\rangle^{<u}_{n-\deg(u^i)}. $$
If  $n= l\cdot \deg(u)$, then $P_l=\{1\}$ and $\lw(f) =u^l$ by Lemma \ref{lex-order-Lyndon}, which contradicts $l<h_I(u)$. Thus we must have $n> \deg(u^l)$, and  the words $w_l$ are all of positive degree $n-\deg(u^l)$.

Now by Proposition \ref{comultiplication},
\begin{eqnarray*}
\Delta(f) &\in & \sum_{w_l\in P_l} \gamma_{w_l} ~ \Delta([w_l]) \cdot \Delta ([u]^{l}) + \sum_{0\leq i < l}  \sum_{w_i\in P_i} \gamma_{w_i} ~ \Delta([w_i]) \cdot \Delta ([u]^{i}) + \big (k\langle X\rangle \otimes k\langle X\rangle \big )_{<n}  \\
&\subseteq& \sum_{w_l\in P_l} \gamma_{w_l} ~ [w_l]  \otimes [u]^l ~ + ~ \Big ( k\langle X\rangle^{<u} \otimes k\langle X\rangle_{\geq 1}^{<u}\cdot [u]^l \Big )_n \\
&& + \sum_{0\leq i<l} \big ( k\langle X\rangle^{\leq u} \otimes k\langle X\rangle^{<u}\cdot [u]^i \big)_n + \big (k\langle X\rangle \otimes k\langle X\rangle \big )_{<n}.
\end{eqnarray*}
By the induction hypothesis, $U_{\leq \deg(u^l)}$ is a complement of $I \cap k\langle X\rangle_{\leq \deg(u^l)}$ in $k\langle X\rangle_{\leq \deg(u^l)}$ with a basis  $\{~ [w]  ~| ~ w\in  \C_I^{\leq \deg(u^l)} ~\}$.
So one may define a linear functional $\rho:k\langle X\rangle\to k$ as
\begin{itemize}
\item $ \rho(I)=\rho(k\langle X\rangle_{>  \deg(u^l)}) =  0$, $ \rho([w])=0$ for $w\in \C_I^{\leq  \deg(u^l)} \backslash \{ u^l \}$ and  $\rho([u]^l) =1$.
\end{itemize}
Then by Lemma \ref{killing-functional},
$$
\rho \big( \sum_{i=0}^{l-1} k\langle X\rangle^{<u} [u]^i \big) \subseteq \rho \Big( \sum_{i=0}^{l-1} (k\langle X\rangle^{<u} [u]^i)_{\leq \deg(u^l)} + k\langle X\rangle_{>  \deg(u^l)} \Big ) =0.$$
It is easy to see $\rho(k\langle X\rangle_i) =0$ for $i\neq \deg(u^l)$, so
$$(\id\otimes \rho) \Big ( (k\langle X\rangle \otimes k\langle X\rangle)_{<n} \Big ) \subseteq k\langle X\rangle_{< n-  \deg(u^l)}.$$
Applying  the linear map $\id\otimes \rho: k\langle X\rangle \otimes  k\langle X\rangle  \to k\langle X\rangle$ to $\Delta(f)$, one gets
\[
\tilde{f}:= (\id\otimes \rho) \Big(\Delta(f) \Big)  \in \sum_{w_l\in P_l} \gamma_{w_l} ~ [w_l]  + k\langle X\rangle_{< n- \deg(u^l)}.
\]
The leading word of $\tilde{f}$ is then contained in $P_l$, so it is $I$-irreducible. Therefore $\tilde{f} \not \in I$.
But $\Delta(f)\in I \otimes  k\langle X\rangle  +  k\langle X\rangle \otimes I$. It  follows that $\tilde{f} \in I$, a contradiction.
\end{proof}

\begin{lemma}\label{triangular-soft}
Let $I$ be a proper ideal of $k\langle X\rangle$. Assume that there exists a  triangular comultiplication  $\Delta$ on $k\langle X\rangle$ such that $\Delta(I) \subseteq k\langle X\rangle \otimes I + I \otimes k\langle X\rangle$. Then for every  Lyndon word $v$ with finite height $n$,
\[
[v]^n +I \in k\langle X\rangle^{<v}_{\deg(v^n)} + k\langle X\rangle_{<\deg(v^n)} +I.
\]
\end{lemma}

\begin{proof}
Suppose that there is a Lyndon word $v$ of finite height $n$ such that the required containment does not hold. So $v^n\not\in \mathcal{C}_I$. By  Lemma \ref{standard-basis-2} and Lemma \ref{quasi-Lie-basis}, $[v]^n+I$ can be expanded uniquely as  $$[v]^n +I=\sum\nolimits_{w\in \mathcal{C}_I^{\deg(v^n)}} \lambda_{w}\cdot([w]+I) + \sum\nolimits_{w'\in \mathcal{C}_I^{<\deg(v^n)}} \lambda_{w'}\cdot([w']+I).$$
By assumption, there exists $w\in \mathcal{C}_I^{\deg(v^n)}$ such that $\lambda_w\neq 0$. Let $u_0$ be the largest Lyndon word that occurs in the Lyndon factorization of any $w\in \mathcal{C}_I^{\deg(v^n)}$ such that $\lambda_w\neq 0$. Note that the powers of $u_0$ necessarily occurs at the end of the Lyndon factorization of these $w\in \mathcal{C}_I^{\deg(v^n)}$ such that $\lambda_w\neq 0$.
 Let $l$ be the maximal occurrences of $u_0$ in the Lyndon factorization of any $w\in \mathcal{C}_I^{\deg(v^n)}$ such that $\lambda_w\neq 0$.
Then $u_0\in \mathcal{N}_I$, $1\leq l< h_I(u_0)$,  and $[v]^n$ can be expressed as following:
\[
 [v]^n \in \sum_{i=0}^{l}  \sum_{w_i\in Q_i} \gamma_{w_i} [w_i] [u_0]^{i} + k\langle X\rangle_{<\deg(v^n)} +I, \quad   \gamma_{w_i} \in k,
\]
where $\gamma_{w_i}\neq 0$ and
$$Q_i \subseteq   \{~w \in \C_I^{\deg(v^n)- \deg(u_0^i)}~|~ \text{Lyndon words in the Lyndon decomposition of $w$ are } <_{\lex} u_0 ~\}.$$
Clearly, $Q_l\neq \emptyset$ by construction. By the assumption, it is necessary that $$v\leq_{\lex} u_0,$$ or otherwise $[v]^n+I$ would have the required containment. Also, since $w_iu_0^i\in \mathcal{C}_I$ for all $i\leq l$ and $w_i\in Q_i$ it follows by Lemma \ref{bracketing-leading} that
\begin{equation*}\label{lll}
w_iu_0^i = \lw([w_i][u_0]^i) <_{\lex} \lw([v]^n) = v^n \tag{*}
\end{equation*}
for all $i\leq l$ and $w_i\in Q_i$. Therefore, there is $g\in I$ with $\lw(g) =v^n$ and
\[
g\in  [v]^n - \sum_{i=0}^{l}  \sum_{w_i\in Q_i} \gamma_{w_i} [w_i] [u_0]^{i} + k\langle X\rangle_{<\deg(v^n)}.
\]
Now we claim that
\[\deg(v^n)>\deg(u_0^l) \quad \text{ or equivalently }\quad 1\not\in Q_l.\] Indeed, if it is not so, then $\deg(v^n)=\deg(u_0^l)$ and $u_0>_{\lex} v$, and then by Lemma \ref{lex-order-Lyndon} one would have $w_lu_0^l = u_0^l>_{\lex} v^n$, which contradicts the relation (\ref{lll}) obtained above.

By Proposition \ref{comultiplication},
\begin{eqnarray*}
\Delta(g) &\in& \Delta([v]^n) - \sum_{0\leq i \leq l}  \sum_{w_i\in Q_i} \gamma_{w_i} ~ \Delta([w_i]) \cdot \Delta ([u_0]^{i}) + \Delta\Big(k\langle X\rangle_{<\deg(v^n)}\Big) \\
&\subseteq& \sum_{p=0}^n \binom{n}{p} ~  [v]^p \otimes [v]^{n-p} -  \sum_{w_l\in Q_l} \gamma_{w_l} ~ [w_l]  \otimes [u_0]^l \\
  && +  \sum_{r, s ~ \geq~ 0 \atop r+s ~<~ n } \Big( k\langle X \rangle_{\geq 1}^{< v} \cdot  [v]^r   \otimes k\langle X\rangle_{\geq 1}^{< v} \cdot [v]^s \Big )_{\deg(v^n)}    +  \Big ( k\langle X\rangle^{<u_0} \otimes k\langle X\rangle_{\geq 1}^{<u_0}\cdot [u_0]^l \Big)_{\deg(v^n)}  \\
&&   + \sum_{0\leq i<l} \big ( k\langle X\rangle^{\leq u_0} \otimes k\langle X\rangle^{<u_0}\cdot [u_0]^i \big)_{\deg(v^n)} +(k\langle X\rangle\otimes k\langle X\rangle)_{<\deg(v^n)}.
\end{eqnarray*}
By Lemma \ref{standard-basis-2} and Lemma \ref{quasi-Lie-basis}, one may define a linear functional $\rho: k\langle X\rangle \to k$  as follows.
\begin{itemize}
\item $\rho(I)= \rho(k\langle X\rangle_{>\deg(u_0^l)}) =0$,  $\rho([w])=0$ for $w\in \C_I \backslash \{ u_0^l \}$ and  $\rho([u_0]^l) =1$.
\end{itemize}
By Lemma \ref{killing-functional} and Lemma \ref{quasi-Lie-basis}, one has $$\rho(\sum_{i=0}^{l-1}k\langle X\rangle^{<u_0}\cdot [u_0]^i) =0.$$ In particular, $\rho(k\langle X\rangle^{<u_0})=0$. If $v<_{\lex} u_0$ then clearly $\id\otimes \rho$ must kill the terms  except the second and sixth ones in the expression of $\Delta(g)$ as above, and whence
\begin{equation*}
(\id\otimes \rho) \Big(\Delta(g) \Big)  \in - \sum_{w_l\in Q_l} \gamma_{w_l} ~ [w_l] + k\langle X\rangle_{< \deg(v^n)- \deg(u_0^l)};
\end{equation*}
and if $v=u_0$ then $v$ is $I$-irreducible and so $n = h_I(v)=h_I(u_0)>l$, and one have  that
\[
(\id\otimes \rho) \Big(\Delta(g) \Big)  \in
 \binom{n}{n-l} \cdot [v]^{n-l} - \sum_{w_l\in Q_l} \gamma_{w_l} [w_l]+ k\langle X\rangle_{<\deg(v^{n-l})}.
\]
The leading word of $(\id\otimes \rho) \Big(\Delta(g) \Big)$ is obviously contained in $\{v^{n-l}\} \cup Q_l \subseteq \mathcal{C}_I$ in both cases, so
 $$(\id\otimes \rho) \Big(\Delta(g) \Big) \not \in I.$$
But $\Delta(g)\in I\otimes k\langle X\rangle + k\langle X\rangle\otimes I$, so $(\id\otimes \rho) \Big(\Delta(g) \Big) \in I$, a contradiction.
\end{proof}

\begin{lemma}\label{height}
Let $I$ be a proper ideal of $k\langle X\rangle$ and $u\in \N_I$ an $I$-irreducible Lyndon word. Assume  there is a  triangular comultiplication  $\Delta$ on $k\langle X\rangle$ such that $\Delta(I) \subseteq k\langle X\rangle \otimes I + I \otimes k\langle X\rangle$.
\begin{enumerate}
\item If $k$ is of characteristic $0$ then $h_I(u) = \infty$.
\item If $k$ is of characteristic $p>0$ then either $h_I(u) = \infty$ or $h_I(u) =p^s$ for some $s\geq1$.
\end{enumerate}
\end{lemma}

\begin{proof}
Suppose that $n=h_I(u)<\infty$. By Lemma \ref{triangular-soft}, there is a polynomial $g\in I$ such that   either
$$g\in [u]^n + k\langle X\rangle_{<\deg(u^n)},$$
or
\[
g \in  [u]^n + \sum_{i=0}^l  \sum_{w_i\in Q_i} \gamma_{w_i} [w_i] [u_0]^{i} + k\langle X\rangle_{<\deg(u^n)}, \quad \gamma_{w_i} \in k,
\]
with $\lw(g)= u^n$, where  $u_0\in \N_I$, $u_0<_{\lex} u$, $h_I(u_0) > l\geq 1$ and
$$Q_i= \{~w \in \C_I^{\deg(u^n)- \deg(u_0^i)}~|~ \text{Lyndon words in the Lyndon decomposition of $w$ are } <_{\lex} u_0 ~\}$$
and $\gamma_{w_l}\neq 0$ for some $w_l\in Q_{l}$. We only deal with
the second case (the first one is similar and much easier). By Proposition \ref{comultiplication},
\begin{eqnarray*}
\Delta(g) &\in& \Delta([u]^n) + \sum_{0\leq i \leq l}  \sum_{w_i\in Q_i} \gamma_{w_i} ~ \Delta([w_i]) \cdot \Delta ([u_0]^{i}) + \big (k\langle X\rangle \otimes k\langle X\rangle \big )_{<\deg(u^n)} \\
&\subseteq&  \sum_{0\leq p \leq n} \binom{n}{p} ~  [u]^p \otimes [u]^{n-p}   \\
&&  +  \sum_{i, j ~ \geq~ 0 \atop i+j ~<~ n }  \big( k\langle X \rangle^{< u}\cdot  [u]^i   \otimes k\langle X\rangle^{< u} \cdot [u]^j \big)_{\deg(u^n)} + \big (k\langle X\rangle \otimes k\langle X\rangle \big )_{<\deg(u^n)}.
\end{eqnarray*}
By Lemma \ref{quasi-Lie-basis},  one may define two linear functionals $\phi, ~ \psi: k\langle X\rangle \to k$ as follows.
\begin{itemize}
\item $\phi(I ) =\phi(k\langle X\rangle_{>\deg(u)}) =0$, $\phi([w]) =0$ for $w \in \C_I \backslash \{ u \}$ and $\phi([u]) =1$;
\item $\psi(I) = \psi( k\langle X\rangle_{>\deg(u^{n-1})})=0$, $\psi([w]) =0$ for $w\in \C_I \backslash \{ u^{n-1} \}$ and  $\psi([u]^{n-1})=1$.
\end{itemize}
By Lemma \ref{killing-functional},
$$
(\phi\otimes \psi) \Big(  \sum_{i,j ~ \geq~ 0 \atop i+j ~<~ n }  \big( k\langle X \rangle^{< u}\cdot  [u]^i   \otimes k\langle X\rangle^{< u} \cdot [u]^j \big)_n  \Big ) =0.$$
It is also easy to see $\phi(k\langle X\rangle_i) =0$ for $i\neq \deg(u)$ and $\psi(k\langle X\rangle_j) =0$ for $j\neq  \deg(u^{n-1})$, so
$$(\phi\otimes \psi) \Big ((k\langle X\rangle \otimes k\langle X\rangle)_{<\deg(u^n)} \Big ) =0.$$
Applying the linear map  $\phi\otimes \psi:k\langle X\rangle \otimes  k\langle X\rangle  \to k$ to $\Delta(g)$, one gets
\begin{equation}\label{equality-formula}
0 = (\phi\otimes \psi )(\Delta(g)) =n \cdot 1_k. \tag{PI}
\end{equation}
Here the first equality holds because $\Delta(g)\in I \otimes  k\langle X\rangle  +  k\langle X\rangle \otimes I$.

When $k$ is of characteristic $0$ then we have got a contradiction (\ref{equality-formula}), and so in this case one must have $h_I(u) = \infty$. This proves (1).

To see (2), we continue to assume $n=h_I(u)<\infty$. By  (\ref{equality-formula}), we may write $n=p^st$ for some $s\geq 1$ and $t$ a positive number that is not divisible by $p$. It remains to show $t=1$. So assume $t>1$. Define two linear functionals  $\phi', ~ \psi': k\langle X\rangle \to k$ as follows.
\begin{itemize}
\item $\phi'(I ) = \phi'(k\langle X\rangle_{>\deg(u^{p^s})})=0$, $\phi'([w]) =0$ for $w \in \C_I \backslash \{u^{p^s} \}$ and $\phi'([u]^{p^s}) =1$;
\item $\psi'(I) = \psi'(k\langle X\rangle_{>\deg(u^{p^s(t-1)})})=0$, $\psi'([w]) =0$ for $w\in \C_I \backslash \{u^{p^s(t-1)} \}$ and  $\psi'([u]^{p^{s}(t-1)})=1$.
\end{itemize}
By  Lemma \ref{killing-functional},
\[
0= (\phi'\otimes \psi')(\Delta(g)) = \binom{p^st}{p^s}\cdot 1_k.
\]
But $p$ doesn't divide $\binom{p^st}{p^s}$, a contradiction.
\end{proof}

\begin{proof}[Proof of Proposition \ref{quasi-Lie-imply-LK}]
Part (1)  follows by an induction on $v$ with respect to the graded lex order and Lemma \ref{triangular-soft}. Part (3) can be read directly from Lemma \ref{quasi-Lie-basis} and Lemma \ref{height} (1).

We show Part  (2) by the induction on the length of $u$.  If $|u|=1$, then $\sh(uv)=(u,v)$ and $[uv]=[u][v]- [v][u]$. If $uv$ is $I$-irreducible, $[uv]\in k\langle X|I\rangle_{\deg(uv)}^{\leq uv}$. If $uv$ is $I$-reducible, by Part (1) one obtains
\[
 [uv] \in k\langle X|I\rangle_{\deg(uv)}^{\leq uv} +k\langle X|I\rangle_{<\deg(uv)} + I.
\]

Now suppose $|u|>1$. If $\sh(uv) =(u,v)$ then the containment holds  by the same discussion as above. Otherwise,  Proposition \ref{fact-Lyndon} (L3, L4) tells us that $u_L>_{\rm lex} u_R>_{\rm lex} v$. One has
\begin{eqnarray*}
[[u], [v]] &=& [[u_L],[[u_R],[v]]] -  [u_R]\cdot [[u_L], [v]]+ [[u_L],[v]]\cdot [u_R] \\
&\in& [[u_L], k\langle X\rangle_{\deg(u_Rv)}^{\leq u_Rv}] + [u_R]\cdot k\langle X\rangle_{\deg(u_Lv)}^{\leq u_Lv} + k\langle X\rangle_{\deg(u_Lv)}^{\leq u_Lv} \cdot [u_R] +k\langle X\rangle_{<\deg(uv)} +I \\
&\subseteq & k\langle X\rangle_{\deg(uv)}^{\leq uv} +k\langle X\rangle_{<\deg(uv)} +I \\
&= & k\langle X|I\rangle_{\deg(uv)}^{\leq uv} +k\langle X|I\rangle_{<\deg(uv)} +I.
\end{eqnarray*}
Here, the containment is by the induction hypothesis;
the last equality is by Part (1) and Lemma \ref{standard-basis-2};
and the inclusion is by Lemma \ref{commutator-subalgebra} and the following two observations. The first one is $u_R<_{\lex} uv$ by Lemma \ref{lex-order}; and the second one is $u_Lv<_{\lex} uv$ by Proposition \ref{fact-Lyndon} (L2).
\end{proof}

\begin{remark}
By \cite[Proposition 6.4]{Zh}, Lemma \ref{quasi-Lie-basis}, Lemma \ref{height} (2) and the first several lines of the argument for Theorem \ref{structure-PBW-generator} (which are independent of the characteristic of $k$), it is easy to reclaim the well-known fact that every finite dimensional connected Hopf algebra over a field of characteristic $p$ is of dimension a power of $p$ (see \cite[Proposition 1.1 (1)]{Masu}).
\end{remark}

\vskip7mm

\noindent{\it Acknowledgments.}
The manuscript was completed during the visit of the first-named author to Hangzhou normal University on March 2019. He would like to thank Jiwei He for helpful discussions, and thank him and Hangzhou normal University for hospitality.
G.-S. Zhou is supported by the NSFC (Grant Nos. 11601480 \& 11871186) and K.C. Wong Magna Fund in Ningbo University;  Y. Shen is supported by the NSFC (Grant No. 11701515) and the Fundamental Research Funds of Zhejiang Sci-Tech University (Grant No. 2019Q071); D.-M. Lu is supported by the NSFC (Grant No. 11671351). The authors thank the referee for his/her careful reading and valuable comments, in particular for the suggestions on some facts of connected Hopf algebras which are dealt in Remark \ref{remark-commutative} and Remark \ref{remark-finiteness}.

\end{document}